\documentclass{article}[12pt]
\usepackage[dvips]{graphics,graphicx,epsfig}
\usepackage{euscript}
\usepackage{amsfonts}
\usepackage{amsmath}
\usepackage{amssymb}
\usepackage{graphicx}
\usepackage{verbatim}
\usepackage{cleveref}
\usepackage{enumitem}
\usepackage{subfig}
\usepackage[dvipsnames]{xcolor}
\usepackage{colortbl}
\usepackage{bm}
\usepackage{eurosym}
\usepackage[colorinlistoftodos]{todonotes}

\usepackage[framemethod=tikz]{mdframed}
\usepackage{amsthm}
\usetikzlibrary{arrows,intersections,chains}
\usetikzlibrary{positioning}
\usepackage{geometry}
\newtheorem{proposition}{Proposition}
\newtheorem{lemma}{Lemma}
\newtheorem{definition}{Definition}

\newtheorem{remark}{Remark}[section]

\newcommand{\tcb}{\textcolor{blue}}
\newcommand{\bd}{\bm}
\newcommand{\mc}{\mathcal}
\renewcommand{\P}{\mathcal{P}}

\newcommand{\ul}{\underline}
\newcommand{\ol}{\overline}
\newcommand{\xx}{\bm{x}}

\newcommand{\rr}{\mathbb{R}}
\newcommand{\eqd}{:=}
\newcommand{\SW}{\text{S\hspace{-1pt}W }}
\newcommand{\nm}{_{nm}}
\newcommand{\mn}{_{mn}}
\newcommand{\N}{\mathcal{N}}
\newcommand{\oG}{\ol{G}}
\newcommand{\uG}{\ul{G}}
\newcommand{\GNE}{\text{GNE}}

\renewcommand{\L}{\mathcal{L}}
\title{Peer-to-Peer Electricity Market Analysis: From Variational to Generalized Nash Equilibrium}

\author{H\'el\`ene Le Cadre\thanks{VITO/EnergyVille research center, Thor scientific park, Genk, Belgium, email: helene.lecadre@energyville.be} \and Paulin Jacquot\thanks{EDF R\&D, OSIRIS, Inria, CMAP Ecole polytechnique, CNRS, France, email: paulin.jacquot@polytechnique.edu} \and Cheng Wan\thanks{EDF R\&D, OSIRIS, France \& RIIS, Shanghai University of Finance and Economics, email: cheng.wan.2005@polytechnique.org} \and Cl\'emence Alasseur\thanks{EDF R\&D and FIME, Energy Markets research center, France, email: clemence.alasseur@edf.fr, the author's research is part of the ANR project PACMAN (ANR-16-CE05-0027)}}

\date{\today}

\tikzset{
    subnodesIEEE/.pic={
       
\draw [name path= line1, line width=1mm, black ] (1.75,-8.5)-- (3.75,-8.5)  node[below ,midway] (anchor1) {} 
node[above ,left=0.8cm] (anchor1v2) {}
node[above ,left=0.4cm] (anchor1v5) {}
node[above ,left] (anchor1v7) {}
node[above ,right=0.2cm]  {1};

\draw [name path= line2, line width=1mm, black ] (8.1,-15.2) node[below , right=0.1cm] (anchor2) {}-- (11.7,-15.2) 
    node[above ,right=0.2cm]  {2}
   node[below ,left=1cm] (anchor2v1) {} 
      node[below ,left=0.2cm] (anchor2v3) {} 
            node[below ,midway] (anchor2v6) {} 
            node[below ,left=0.4cm] (anchor2v4) {}  ;

\draw [name path= line3, line width=1mm, black ] (13.9,-15.2) node[below] (anchor3) {} -- (16.75,-15.2) node[above ,right=0.2cm]  {3}
node[below ,left=1cm] (anchor3v2) {} 
node[below ,midway] (anchor3v4) {};

\draw [name path= line4, line width=1mm, black ] (14.1,-11.5)-- (16.4,-11.5) 
node[above ,right=0.2cm]  {4}
node[above ,left=0.7cm] (anchor4v2) {} 
node[above ,left=0.5cm] (anchor4v3) {}
node[above ,left=0.9cm] (anchor4v5) {}
node[above ,left=0.25cm] (anchor4v7) {}
node[above, left=0.2cm] (anchor4) {};

\draw [name path= line5, line width=1mm, black ] (8.8,-11.5)-- (10.8,-11.5)
node[above ,right=0.2cm]  {5} 
node[above ,left=0.7cm] (anchor5v1) {}
node[above ,left=0.4cm] (anchor5v2) {}
node[above ,left=0.2cm] (anchor5v4) {};

\draw [name path= line6, line width=1mm, black ] (8.8,-6.4)-- (10.8,-6.4)
node[above ,right=0.2cm]  {6}
node [above, left=0.15cm] (anchor6) {}
node[above, left=0.9cm] (anchor6v1) {} 
node[above, left=0.4cm] (anchor6v2) {}
node[above, left=0.3cm] (anchor6v11) {}
node[above, left=0.6cm] (anchor6v13) {};

\draw [name path= line7, line width=1mm, black ] (14.8,-8.2)-- (16.2,-8.2) 
node[above ,right=0.2cm]  {7}
node[above, left=0.2cm] (anchor7v4) {} 
node[above, left=0.5cm] (anchor7v1) {}
node[above, left=0.1cm] (anchor7v8) {};

\draw [name path= line8, line width=1mm, black ] (17.4,-7.8)-- (17.4,-6.3) 
node[above =0.2cm,]  {8}
node[above, midway] (anchor8v7) {}
node[above, left] (anchor8) {} ;

\draw [name path= line9, line width=1mm, black ] (14.1,-6)-- (16.1,-6) 
node[above ,right=0.2cm]  {9}
node[above, left=0.6cm] (anchor9){}
node[above, left=0.8cm] (anchor9v4) {}
node[above, left=0.5cm] (anchor9v7) {}
node[above, left=0.2cm] (anchor9v14) {}  ;

\draw [name path= line10, line width=1mm, black ] (13.3,-3.9)-- (14.7,-3.9)
node[above ,right=0.0cm]  {10}
node[above, left=0.6cm] (anchor10){}
node[above, left=0.2cm] (anchor10v9){}
node[above, midway] (anchor10v11) {};

\draw [name path= line11, line width=1mm, black ] (9.9,-3.9)-- (11.2,-3.9) 
node[above ,right=0.2cm]  {11}
node[above, left= 0.2cm] (anchor11v10) {}

;
\draw [name path= line12, line width=1mm, black ] (6.3,-3.9)-- (7.6,-3.9) 
node[above ,right=0.2cm]  {12}
node [above, midway] (anchor12) {} 
node [above, left=0.2cm] (anchor12v13) {}  ;

\draw [name path= line13, line width=1mm, black ] (8.6,-2.5)-- (10.6,-2.5)
node[above ,right=0.2cm]  {13}
node[above, midway] (anchor13){}
node[above, left=0.4cm] (anchor13v6) {} 
node[above, left=0.2cm] (anchor13v14) {} 
node[above, left=0.8cm] (anchor13v12) {} ;

\draw [name path= line14, line width=1mm, black ] (13.9,-2.5)-- (15.9,-2.5)
node[above ,right=0.2cm]  {14}
node[above, midway] (anchor14){}
node[above, left=0.2cm] (anchor14v9) {}
node[above, left=0.7cm] (anchor14v13) {} ;

\def\nd{0.3}
\def \idn{0.15} 
\node[draw,circle,fill=none,font=\small,line width=0.5pt, above= \nd cm of anchor1, node distance= 0.2cm] (G0) {G};

\draw (G0.south)-| (anchor1) ;

\node[draw,circle,fill=none,line width=0.5pt, below= \nd cm of anchor2] (c2) {C};

\draw (c2)-| (anchor2) ;
\node[draw,circle,fill=none,line width=0.5pt, right= \idn cm of c2] (g2) {G};
\draw (g2)--++(0,\nd cm) |- (anchor2v1) ;

\node[draw,circle,fill=none,line width=0.5pt, right= \idn cm of g2] (g4) {G};
\draw (g4)--++(0,\nd cm) |- (anchor2v1) ;

\node[draw,circle,fill=none,line width=0.5pt, below= \nd cm of anchor3] (c3) {C};
\draw (c3)-| (anchor3) ;
\node[draw,circle,fill=none,line width=0.5pt, right= \idn cm of c3] (g6) {G};
\draw (g6)--++(0,\nd cm) |- (anchor3v4) ;

\node[draw,circle,fill=none,line width=0.5pt, right= \idn cm of g6] (g7) {G};
\draw (g7)--++(0,\nd cm) |- (anchor3v4) ;

\node[draw,circle,fill=none,line width=0.5pt, below= \nd cm of anchor4] (c4) {C};
\draw (c4)--++(0,\nd cm) |- (anchor4) ;

\node[draw,circle,fill=none,line width=0.5pt, above= \nd cm of anchor5v1] (c5) {C};
\draw (c5.south)--++(0,\nd cm) |- (anchor5v1) ;

\node[draw,circle,fill=none,line width=0.5pt, right= \nd cm of anchor8] (g10) {G};
\draw (g10)--   (anchor8) ;

\node[draw,circle,fill=none,line width=0.5pt, below= \nd cm of g10] (g11) {G};

\draw (g11)-| (anchor8) ;

\node[draw,circle,fill=none,line width=0.5pt, above = \nd cm of anchor9] (c9) {C};
\draw (c9)--   (anchor9) ;
\node[draw,circle,fill=none,line width=0.5pt, above = \nd cm of anchor14] (c14) {C};
\draw (c14)--   (anchor14) ;
\node[draw,circle,fill=none,line width=0.5pt, below = \nd cm of anchor10] (c10) {C};
\draw (c10)--   (anchor10) ;

\node[draw,circle,fill=none,line width=0.5pt, below= \nd cm of anchor6] (g15) {G};

\draw (g15)-| (anchor6) ;
\node[draw,circle,fill=none,line width=0.5pt, left= \idn cm of g15] (c6) {C};
\draw (c6)-- ++ (0, 1) |- (anchor6) ;
\node[draw,circle,fill=none,line width=0.5pt, below= \nd cm of anchor11v10] (c11) {C};
\draw (c11)-| (anchor11v10) ;
\node[draw,circle,fill=none,line width=0.5pt, above= \nd cm of anchor12] (c12) {C};
\draw (c12.south)-| (anchor12) ;
\node[draw,circle,fill=none,line width=0.5pt, above= \nd cm of anchor13] (c13) {C};
\draw (c13.south)-| (anchor13) ;
    }}

\begin{document}

\maketitle  

\begin{abstract}
We consider a network of prosumers involved in peer-to-peer energy exchanges, with differentiation price preferences on the trades with their neighbors, and we analyze two market designs: \textbf{(i)} a centralized market, used as a benchmark, where a global market operator optimizes the flows (trades) between the nodes, local demand and flexibility activation to maximize the system overall social welfare; \textbf{(ii)} a distributed peer-to-peer market design where prosumers in local energy communities optimize selfishly their trades, demand, and flexibility activation.
We first characterize the solution of the peer-to-peer market as a Variational Equilibrium  and prove that the set of Variational Equilibria coincides with the set of social welfare optimal solutions of market design \textbf{(i)}. 
We give several results that help understanding the structure of the trades at an equilibrium or at the optimum. We characterize the impact of preferences on the network line congestion and renewable energy waste under both designs.
We provide a reduced example for which we give the set of all possible generalized equilibria, which enables to give an approximation of the price of anarchy. We provide a more realistic example which  relies on the IEEE 14-bus network, for which we can simulate the trades under different  preference prices. Our analysis shows in particular that the  preferences have a large impact on the structure of the trades, but that one equilibrium (variational) is optimal.
\end{abstract}


\section{Introduction} \label{sec:intro}

New regulations are restructuring electricity markets in order to build the grid of the future. Instead of a centralized market design where all the operations have been managed by a global  central market operator \cite{madani, schweppe, stoft}, 
new \emph{decentralized}  models emerge. These models  involve local energy communities which can trade energy, either  by the intermediate of a global market operator \cite{lecadreCEJOR2018}, or in a peer-to-peer setting \cite{parag, sousa}.
Peer-to-peer energy trading allows flexible energy trades between peers, where, for instance, local prosumers exchange between them energy surplus from multiple small-scale distributed energy resources (DERs)  \cite{liu2017, long2017}.

Significant value is brought to the power system by coordinating local renewable energy source (RES)-based generators and DERs to satisfy the demand of local energy communities, since it decreases the need for investment in conventional generations and transmission networks. Also, thanks to the decreasing feed-in-tariffs, using RES-based generations on site (e.g., at household level, within the microgrid) is more attractive than feeding it into the grid, because of the difference between electricity selling and buying prices \cite{long2017}. Peer-to-peer energy trading encourages the use of surplus energy within local energy communities, resulting in significant cost savings even for communities with moderate penetration of RES \cite{long2017}. 

In practice, the radial structure of the distribution grid calls for hierarchical market designs, involving transmission and distribution network operators \cite{lecadreEJOR2019}. Nevertheless, various degrees of coordination can be envisaged: full coordination organized by a global market operator (transmission network operator), bilateral contract networks \cite{morstyn2018}, fully decentralized market designs allowing peer-to-peer energy trading between the prosumers in a distributed fashion \cite{moret2018, sorin2018} or, still, within and between coalitions of prosumers, called community or hybrid peer-to-peer \cite{moret}. A community-based organization involves a community manager which organizes trades among the community and is in charge of the interactions with the rest of the market. A distributed market structure exists when the decentralized elements explicitly share, in a peer-to-peer fashion, local information, resulting in a system in which all the elements may not have access to the same level of information. This information asymmetry might create differences in valuations of the traded resource (e.g., price arbitrage) and result in market imperfections, implying that the prices associated with the bilateral trading of resource allocation between couples of agents do not coincide. This price gap can be interpreted as a bid-ask spread due to a lack of liquidity in the market \cite{oggioni}.

Energy exchange between production units and  local demand of energy communities are formulated as a symmetric assignment problem. Its solution relies on two main streams in the literature. The first stream deals with matching models which put in relation RES-based generators and consumers by the intermediate of a platform, with various consumers classes and different possible objective functions for the platform operator \cite{liu2017}. The second stream combines multi-agent modeling, as well as classical distributed optimization algorithms which are applied to solve the assignment problem in real-time \cite{moret2018, morstyn2018_multiclass, sorin2018}. Auctions theory can be used, in addition, to schedule the DER commitment in day-ahead. 

\subsection{Matching Models for Peer-to-Peer Energy Trading} 

In the energy sector, peer-to-peer energy trading is a novel paradigm of power system operation. There, prosumers provide their own energy from solar panels, storage technologies, demand response mechanisms, and they exchange energy with one another in a distributed fashion. Zhang et al. provide in \cite{zhang2016} an exhaustive list of projects and trails all around the world, which build on new innovative approaches for peer-to-peer energy trading. A large part of these projects rely on platforms, understood as two-sided markets, that match RES-based generators and consumers according to their preferences and locality aspects (e.g. Piclo in the UK, TransActive Grid in Brooklyn, US, Vandebron in the Netherlands, etc.). In the same vein, cloud-based virtual market places, which deal with excess generation within microgrids, are developed by PeerEnergyCloud and Smart Watts in Germany. Some other projects rely on local community
building for investment sharing in batteries, solar PV panels, etc., in exchange for bill reduction or a certain level of autonomy with respect to the global grid (e.g. Yeloha and Mosaic in the US, SonnenCommunity in Germany, etc.). How other components of the platform's design can influence the nature and the preference of the prosumers involved is also studied in the literature. Typical elements of the platform's design are: the impact of pricing mechanism (e.g. setting one common market price versus individual prices per transaction set -- for instance through auction design -- or per class of prosumers), the platform's objective (e.g. maximizing the social welfare versus maximizing the platform's benefit), the influence of the platform's commission per transaction. For example, in \cite{benjaafar}, the authors study the impact of the price of the goods exchanged on the level of collaboration and also on the level of ownership among participants. In \cite{fang}, the impact of different platform's objective functions is analyzed considering  a set of heterogeneous renters and  owners. Dynamic pricing for operations of the platform based on supply and demand ratio of shared RES-based generation is investigated in \cite{liu2017}. Peer-to-peer organizations are also a way to enable small and flexible actors to enter markets by lowering the entrance barrier  \cite{einav2015}. 

Platform design constitutes an active area of research in the literature on two-sided markets
\cite{einav2015, fang}. Three needs are identified for platform deployment. Firstly, it should help buyers and sellers find each other, while taking into account the heterogeneity in their preferences. This requires the platform to find a trade-off between low entry cost
and information retrieval from big, heterogeneous and dynamic information
flows. Buyers' and sellers' search can be performed in a \emph{centralized} (e.g. Amazon,
Uber), effective \emph{decentralized} (e.g. Airbnb, eBay),
or even \emph{fully distributed} (OpenBazaar, Arcade City) manner. Secondly, the platform must
set prices that balance demand and supply, and ensure that prices are set competitively
in a decentralized fashion. Finally, the platform ought to maintain trust in
the market, relying on reputation and feedback mechanisms \cite{foti2018}. Sometimes, supply
might be insufficient so that subsidies needs to be designed to encourage sharing on
the platform \cite{fang}.

\subsection{Distributed Optimization Approaches} 

Computational and communication bottlenecks have largely been alleviated by recent work on distributed and peer-to-peer optimization of large-scale optimal power 
flow \cite{engels2016, kraning2014, peng2014}. Mechanisms for the optimization of a common objective function by a decentralized system are known as decomposition-coordination methods \cite{palomar}. In such methods, a centralized (large-scale) optimization problem is typically split into small-size local optimization problems whose outputs are coordinated dynamically by a central agent (called ``master") so that the overall objective of the system becomes aligned (after a certain number of iterations) with the (large-scale) centralized optimization problem outcome. Following this stream, a consensus-based Alternating Direction Method of Multipliers (ADMM) algorithm is implemented in \cite{morstyn2018_multiclass, sorin2018, smartest} to approximate the optimal solution which maximizes the prosumers social welfare, in a peer-to-peer electricity market. Similar approaches relying on dual decomposition, which 
iteratively solves the problem in a distributed manner with limited information exchange, were implemented for energy trading between islanded microgrids in \cite{gregoratti2016, matamoros2012}. Two main drawbacks of these algorithmic approaches are listed in \cite{sorin2018}: first, they do not take into account the strategic behaviors of the prosumers; second, they are computationally limited, which might constitute a blocking point when studying large-scale peer-to-peer networks. The latter issue is overcome in \cite{moret2018} with an improved consensus algorithm. 

In addition, these distributed-optimization approaches enable incorporating heterogeneous energy preferences of individual prosumers in network management. The added value of multi-class prosumer energy management is evaluated in \cite{morstyn2018_multiclass} for a distribution network that has a ``green prosumer", a ``philanthropic prosumer" and a ``low-income household". Three energy classes are introduced to account for the prosumers' preferences: ``green energy", ``subsidized
energy" and ``grid energy". A platform agent is introduced to act as an auctioneer, allowing energy trading between the prosumers and the wholesale electricity market. The platform agent sets the price of each energy class in the distribution network. The tool of receding-horizon model predictive control is used to provide a real-time
implementation. Consumer preferences are also introduced in \cite{sorin2018} in the form of product differentiation prices. They can either be pushed centrally as dynamic and specific tax payments, or be used to better describe the utility of the consumers who are willing to pay for certain characteristics of trades.

\subsection{Privacy Issues} 

From the perspective of information and communication technology (ICT) , a fully decentralized market design provides a robust framework since, if one node in a local market is attacked or
in case of failures, the whole architecture should remain in place, while information could find other paths to circulate from one point to another, avoiding malicious nodes and corrupted paths. 

From an algorithmic point of view, the implementation of a fully distributed market design might be challenging, since it has to deal with far more complex communication mechanisms than the centralized market design. Efficient communication will allow collaboration among prosumers, so that energy produced by one can be utilized by another in the network. Multiple peer-to-peer communication architectures exist in the literature,
including structured, unstructured and hybrid ones. They are all based on common standards for the communication network operation, which are measured through latency, throughput, reliability and security \cite{jogunola2018}. In addition to the large size of the communication problem, privacy issues may also directly impact the market outcome. Indeed, if prosumers are allowed to keep some private information, then they might not have access to the same level of information, i.e. information asymmetry appears. Since the prosumers' make decisions based on the information at their disposal, such asymmetry can introduce bias in the market outcome. To avoid or, at least, to limit bias introduced in the market outcome while guaranteeing the optimum of the social welfare, various algorithms that preserve local market agents' privacy have been discussed in the literature. For example, the algorithms can require the agents to update no more than their dual variables -- e.g., local prices \cite{engels2016, sorin2018}. Of course, the efficiency of these algorithms depends on the level of privacy defined by the agents as well as which private information could be inferred from the released values.

\subsection{Contributions} 

The peer-to-peer structure  adopted in this paper is different from the approaches involving decomposition-coordination methods. Indeed, we assume that there is no central authority coordinating the exchanges (in quantity, price and information) between the nodes. Within this framework, strategic communication mechanisms can appear, and nodes have the possibility to self-organize into coalitions or local energy communities, as reviewed in \cite{tushar}. With such strategic behaviors, the equilibrium of the peer-to-peer market design might not coincide with the social welfare global optimum achieved with full coordination of the nodes by a ``master"  controlling all the information and decisions, as in \cite{wang2014} where the authors consider a noncooperative game involving storage units. 

In this paper, we first characterize the solution of a peer-to-peer electricity market as a Variational Equilibrium, assuming that all the agents have equal valuation of the price associated with the traded resource. We prove that the set of Variational Equilibria coincides with the set of social welfare optima. However, in a fully-distributed setting, it is very unlikely that each couple of agents coordinate on their valuations of the trading price. As a result, imperfections appear in the market, which we capture by considering Generalized Nash Equilibrium solutions as possible outcomes. 
We characterize analytically the impact of preferences on the network line congestion and energy waste, both under centralized and peer-to-peer market designs. Our results are illustrated in two test cases (a three node network with arbitrage opportunity and the standard IEEE-14 bus network). 
We evaluate the loss of efficiency caused by peer-to-peer market imperfections in the three nodes network, with the Price of Anarchy as a performance measure.
We also evaluate numerically the impact of the differentiation prices by computing the equilibria of our 14 nodes network under different price configurations.
Last, we quantify the impact of privacy on the social welfare at equilibrium by providing an analytic upper bound and evaluating it in our three nodes example. 

\vspace{0.5cm}
The paper is organized as follows. In Section~\ref{sec:model}, we introduce the model of the generalized noncooperative game we consider in this work, and we give our main assumptions.
In Section~\ref{sec:centralized_md}, the centralized market design \textbf{(i)} is formulated and its solutions characterized. We introduce the peer-to-peer market design \textbf{(ii)} in Section~\ref{sec:p2p_md}; its solutions are characterized in terms of Variational Equilibrium and Generalized Nash Equilibria in the presence of market incompleteness. Congestion analysis and performance measure based on the Price of Anarchy are also introduced. These solutions concepts are then applied to two test cases in Section~\ref{sec:test_cases}: a three node toy network and the IEEE 14-bus network. The impact of privacy is quantified in Section~\ref{sec:privacy}, and illustrated on the three node toy network. 

\section*{Notations}

We summarize the main notations used throughout the paper. Vectors and matrices are denoted by \textbf{bold} letters. 

\vspace{0.3cm}

\textbf{Sets}

\vspace{0.2cm}

\begin{tabular}{|c||c|}\hline
$\mathcal{N}$ & Set of $N$ nodes, each one of them being made of an agent (prosumer) \\\hline

$\Omega_n$ & Set of neighbors of $n$\\\hline

$\mathcal{D}_n$ & Agent $n$'s demand set \\\hline

$\mathcal{G}_n$ & Agent $n$'s flexibility activation set \\\hline

$\text{SOL}^{\text{GNEP}}$ & Set of GNE solutions of the peer-to-peer non-cooperative game \\\hline
\end{tabular}

\vspace{0.3cm}

\textbf{Variables}

\vspace{0.2cm}

\begin{tabular}{|c||c|}\hline
$D_n$ & Agent $n$'s demand \\\hline

$G_n$ & Agent $n$'s flexibility activation (micro-CHP, storage facilities, etc.) \\\hline

$\Delta G_n$ & Agent $n$'s random self-generation obtained from RES (solar PV panels) \\\hline

\end{tabular}

\begin{tabular}{|c||c|}\hline

$q_{mn}$ & Quantity exchanged between $n$ and $m$ in the direction from $m$ to $n$ \\\hline

$Q_n$ & Net import at node $n$ \\\hline

$\zeta_{nm}$ & Bilateral trade price between agent $n$ and $m$ \\\hline

$\lambda_n$ & Nodal price at node $n$ \\\hline

$\xi_{nm}$ & Congestion price between nodes $n$ and $m$ \\\hline

$\underline{\mu}_n,\overline{\mu}_n$ & Demand capacity constraint dual variables at node $n$ \\\hline

$\underline{\nu}_n,\overline{\nu}_n$ & Flexibility activation capacity constraint dual variables at node $n$ \\\hline

$\epsilon_{nm}^D, \epsilon_{nm}^G$ & Agent $n$'s biases in the estimation of $m$ demand and RES-based generation \\\hline
\end{tabular}

\vspace{0.3cm}

\textbf{Parameters}

\vspace{0.2cm}

\begin{tabular}{|c||c|}\hline
$\underline{D}_n$ & Lower-bound on demand capacity \\\hline

$\overline{D}_n$ & Upper-bound on demand capacity \\\hline

$\underline{G}_n$ & Lower-bound on flexibility activation capacity \\\hline

$\overline{G}_n$ & Upper-bound on flexibility activation capacity \\\hline

$D_n^{\star}$ & Agent $n$'s target demand \\\hline

$\kappa_{nm}$ & Equivalent interconnection capacity between node $n$ and node $m$ \\\hline

$a_n, b_n, d_n$ & Flexibility activation cost parameters \\\hline

$\tilde{a}_n, \tilde{b}_n$ & Consumer utility parameters \\\hline

$c_{mn}$ & Product differentiation price capturing agent $n$'s trading cost preferences \\\hline

$\delta_{nm}$ & Agent $m$ valuation of $\zeta_{nm}$ \\\hline

$\sigma_{nm}^D, \sigma_{nm}^G$ & Standard deviation of agent $n$'s error in demand and RES forecasts \\\hline
\end{tabular}

\vspace{0.3cm}

\textbf{Functions}

\vspace{0.2cm}

\begin{tabular}{|c||c|}\hline
$C_n$ & Agent $n$'s flexibility activation (production) cost \\\hline 

$U_n$ & Agent $n$'s usage benefit \\\hline
    
$\tilde{C}_n$ & Agent $n$'s total trading cost \\\hline
    
$\Pi_n$ & Agent $n$'s utility function \\\hline
    
$SW$ & Social welfare \\\hline

$F_n$ & Agent $n$'s forecast \\\hline
\end{tabular}

\section{Prosumers and Local Communities} \label{sec:model}

In this section, we define the generic framework of agent (prosumer) interactions, and a stylized representation of the underlying (distribution) graph. We formulate the local supply and demand balancing constraint that holds in each node. To formalize the two market designs \textbf{(i)} and \textbf{(ii)}, we introduce the costs, utility functions, social welfare, private information and main assumptions on which our model relies.

\subsection{Generic Framework}

Let $\mathcal{N}$ be a set of $N$ nodes, each of them representing an agent (prosumer), except the root node $0$ which is assumed to contain only conventional generation. The root node belongs to the set $\mathcal{N}$. It can trade energy with any other node in $\mathcal{N}$. Under this assumption, the distribution network is a radial graph, with the root node being the interface between the local energy communities and the transmission network. Figure~\ref{fig:graphExample}  illustrates such a graph structure.

\begin{figure}[htbp]
\begin{center}
\includegraphics[scale=0.45]{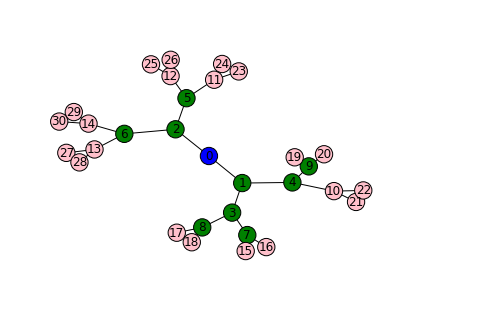}
\caption{\small{Example of a radial network. The root node at the interface of the distribution and transmission networks, can trade energy with any other node in the distribution network. In the distribution network, prosumer nodes organize in local energy communities, trading energy with neighbors inside their local community.}}
\label{fig:graphExample}
\end{center}
\end{figure}

Let $\Omega_n$ be the set of neighbors of $n$, with the structure of a communication network (local energy community). It does not necessarily reflect the grid constraints. As usual, we assume that $n \in \Omega_n$, for all $n \in \mathcal{N}$. In particular, $\Omega_0 := \mathcal{N}\setminus \{0\}$.

In each node $n$, we introduce $\mathcal{D}_n := \{ D_n \in \mathbb{R}_+ | \underline{D}_n \leq D_n \leq \overline{D}_n\}$ as agent $n$'s demand set, with $\underline{D}_n$ and $\overline{D}_n$ being the lower and upper-bounds on demand capacity.

In parallel to the demand-side, we define the self-generation-side by letting $\mathcal{G}_n := \{G_n \in \mathbb{R}_+ | \underline{G}_n \leq G_n \leq \overline{G}_n \}$ be agent $n$'s flexibility activation set, where $\underline{G}_n$ and $\overline{G}_n$ are the lower and upper-bounds on flexibility activation capacity.

The decision variables of each prosumer $n$ are her demand $D_n$, flexibility activation $G_n$, and the quantity exchanged between $n$ and $m$ in the direction from $m$ to $n$, $q_{mn}$, for all $m\in \Omega_n\setminus \{n\}$. If $q_{mn} \geq 0$, then $n$ buys $q_{mn}$ from $m$, otherwise ($q_{mn}<0$) $n$ sells $-q_{mn}$ to $m$. We impose an inequality on the trading reciprocity $q_{mn} \leq -q_{nm}$, meaning that the quantity that $n$ buys from $m$ is smaller than or equal to the quantity that $m$ sells to $n$. 

The difference between the sum of imports and the sum of exports in node $n$ is defined as the net import in that node: $Q_n : = \sum_{m \in \Omega_n} q_{mn}$. Furthermore, each line is constrained in capacity. Let $\kappa_{nm}\in [0 , +\infty[$ be the equivalent interconnection capacity between node $n$ and node $m$, such that $q_{nm} \leq \kappa_{nm}$, $\kappa_{nm} = \kappa_{mn}$ and $\overline{G}_n \leq \kappa_{nm}$. 

RES-based (solar PV panels) self-generation at each node $n$ is modeled as a random variable $\Delta G_n$. Its realization is exogenous to our model.

\subsection{Local Supply and Demand Balancing}\label{subsec:local_balancing}

Local supply and demand equilibrium leads to the following equality in each node $n$ in $\mathcal{N}$:
\begin{eqnarray}\label{eq:supply_demand_eq}
D_n &=& G_n + \Delta G_n + \sum_{m \in \Omega_n}  q_{mn}, \nonumber \\
&=& G_n + \Delta G_n + Q_n. 
\end{eqnarray}

Assuming perfect competition, a Market Operator (MO) maximizes the system social welfare,  defined as the sum of the utilities of all the agents in the system, under a set of operational and power-flow constraints, while checking that supply and demand balance each other at each node of the network. In nodal markets, allocative market efficiency can be achieved by setting (locational marginal) nodal price, $\lambda_n$, equal to the dual variable of the local supply and demand balancing equation \cite{sherali}.

In this paper, we consider a 
\emph{innovative decentralized} marker clearing, 
by comparison with the classical centralized approach, which is used for example in nodal markets. For that purpose, we introduce decentralization in agents' decision-making. This decentralization results firstly from the fact that demands, flexibility activation and trades are defined selfishly by each prosumer in the nodes; secondly from the fact that all the information regarding preferences and private information on target demands and RES-based generations is not available to all the nodes. The decentralized market clearing relies on a peer-to-peer market design, where each agent $n$ computes the Lagrangian variable associated with her (local) supply and demand balancing equation, using the information at her disposal. Dual variables $\lambda_n$ are kept private to agent $n$ and used to compute her bilateral trading prices.  

\subsection{Cost and Usage Benefit Functions}

Flexibility activation (production) cost in node $n$ is modeled as a quadratic function of local activated flexibility, using three positive parameters $a_n$, $b_n$ and $d_n$:
\begin{equation}\label{eq:prod_cost}
C_n(G_n) = \frac{1}{2} a_n G_n^2 + b_n G_n +d_n,
\end{equation}
with $-\frac{b_n}{a_n}\geq \underline{G}_n$.

We make the standard assumption  that self-generation occurs at zero marginal cost.

The usage benefit perceived by agent $n$ is modeled as a strictly concave function of node $n$ demand  \cite{fang}, using two positive parameters $\tilde{a}_n$, $\tilde{b}_n$ and a target demand defined exogenously by agent $n$:
\begin{equation}\label{eq:consumption_cost}
U_n(D_n) = -\tilde{a}_n(D_n-D_n^{\star})^2+\tilde{b}_n.
\end{equation}
 The quantity $-U_n(.)$ can also be considered as the consumption cost of agent $n$ \cite{sorin2018}. 
 As $U_n(.)$ captures a usage benefit, which is interpreted as the comfort perceived by agent $n$, we impose that it always remains non-negative, i.e., 
$\overline{D}_n - \sqrt{\frac{\tilde{b}_n}{\tilde{a}_n}} \leq D_n^{\star} \leq \underline{D}_n + \sqrt{\frac{\tilde{b}_n}{\tilde{a}_n}}$. 

In Section~\ref{sec:privacy}  only, we will assume that the maximum usage benefit is homogeneous among all agents (meaning that the prosumers differentiate only in the distribution of their usage benefit perception around her target demand value). In general, it is different between two prosumers, but will be assumed to be shared publicly. 

We consider that usage benefit vanishes in case of zero demand, i.e., $U_n(0)=0 \Leftrightarrow \tilde{a}_n = \frac{\tilde{b}_n}{(D_n^{\star})^2}, \forall n \in \mathcal{N}$. A fortiori, if the target demand $D_n^{\star}$ is known, coefficient $\tilde{a}_n$ can be inferred.   
 
In this work, we consider that prosumers have preferences on the possible trades with their neighbors. The preferences are modeled with  (product) differentiation prices \cite{sorin2018}: each   agent $n$ has a positive price $c_{nm}>0$ to buy energy to an agent $m$ in her neighborhood $\Omega_n$. The total trading cost function of agent $n$ is denoted by:
\begin{equation}\label{eq:cost_preference}
\tilde{C}_n(\bm{q}_n) = \sum_{m \in \Omega_n,m\neq n} c_{nm} q_{mn}.
\end{equation}
Parameters $c_{nm}$ can model taxes to encourage/refrain the development of certain technologies (micro-CHPs, storage, solar panels) in some nodes. They can also capture agents' preferences to pay regarding certain characteristics of trades (RES-based generation, location of the prosumer, transport distance, size of the prosumer, etc.). If $q_{mn}  >0$ (i.e., $n$ buys $q_{mn}$ to $n$) then $n$ has to pay the cost $c_{nm}q_{mn} >0$. Thus, the higher $c_{nm}$ is, the less interesting it is for $n$ to buy energy from $m$ but the more interesting it is for $n$ to sell energy to $m$. 
On the other side, if $q\mn <0$, then $n$ sends the energy $-q\mn$ and receives the value $-c\nm q\mn >0$ even if $m$ does not accept all this energy (i.e. $q\nm + q\mn < 0$). In that case the remaining power is injected in the network or wasted.

\subsection{Utility Function and Social Welfare}

Agent $n$'s utility function is defined as the difference between the usage benefit resulting from the consumption of $D_n$ energy unit and the sum of the flexibility activation and trading costs. Formally, it takes the form:
\begin{equation} \label{eq:profit}
\Pi_n(D_n, G_n,\bm{q}_n) = U_n(D_n) - C_n(G_n) - \tilde{C}_n(\bm{q}_n),
\end{equation}
where $\bm{q}_n=(q_{mn})_{m \in \Omega_n, m \neq n}$.

We introduce the social welfare as the sum of the utility functions of all the agents in $\mathcal{N}$:
\begin{equation} \label{eq:social_welfare}
SW(\bm{D}, \bm{G},\bm{q}) = \sum_{n \in \mathcal{N}} \Pi_n(D_n, G_n,\bm{q}_n).
\end{equation}

\subsection{Private Information and Assumptions}

There is private information at each node $n$ that can be associated with:
\begin{itemize}
\item $\Delta G_n$, local RES-based generation ;
\item $D_n^{\star}$, target demand ;
\item $C_n(.)$, flexibility activation cost function, more specifically parameters $a_n$, $b_n$, $d_n$ ;
\item $U_n(.)$, usage benefit function, more specifically parameters $\tilde{a}_n$, $\tilde{b}_n$ ;
\item $\tilde{C}_n(.)$, bilateral trade cost function, more specifically parameters $(c_{nm})_{m \in \mathcal{N}\setminus \{n\}}$ .
\end{itemize}

Throughout this article, we will make some assumptions regarding the information available to each agent. We summarize them below:

\begin{description}
\item[Assumption~1] Since the technologies (conventional units, solar PV panels, micro-CHPs, storage facilities, etc.) used by the agents are standardized, we assume that each agent $n$'s production cost $C_n(.)$ and parameters $a_n$, $b_n$, $d_n$ are publicly known by all the agents $m \in \mathcal{N}\setminus \{n\}$.

\item[Assumption~2] The product differentiation prices $(c_{n0})_{n}$, $(c_{0n})_{n}$ for any $n \in \mathcal{N} \setminus \{0\}$ are publicly known by all the agents. In case of taxes, they might be designed by the regulator to impact the energy mix.

\item[Assumption~3] The congestion prices $(\xi_{n0})_{n}$, $(\xi_{0n})_{n}$ for any node $n \in \mathcal{N} \setminus \{0\}$ on the interface lines between transmission network and distribution networks are determined by a dedicated market mechanism and publicly revealed to all the agents.

\end{description}

In a centralized market design, all the private information is reported to the Market Operator (MO). This means that the local target demands $(D_n^{\star})_{n \in \mathcal{N}}$ and RES-based generations $(\Delta G_n)_{n \in \mathcal{N}}$, are known by the MO. In contrast, in a peer-to-peer market design, $D_n^{\star}$ and $\Delta G_n$ are known only by agent $n$.

\section{Centralized Market Design}\label{sec:centralized_md}

The centralized market design is inspired from the existing pool-based markets. The global Market Operator (MO) maximizes the social welfare defined in Equation~(\ref{eq:social_welfare}) under demand capacity constraints (\ref{eq:CM1}) and flexibility activation capacity constraints (\ref{eq:CM2}) in each node, capacity trading flow constraints for each couple of nodes (\ref{eq:CM3}), trading reciprocity constraint (\ref{eq:CM4}) and supply-demand balancing (\ref{eq:CM5}) in each node:
\begin{subequations}
\begin{align}
\max_{\bm{D}, \bm{G}, \bm{q}} \hspace{1cm} &  SW(\bm{D}, \bm{G}, \bm{q}), \nonumber \\
s.t. \hspace{1cm} & \underline{D}_n \leq D_n \leq \overline{D}_n, \forall n \in \mathcal{N}, \hspace{4.5cm} &(\tcb{\underline{\mu}_n}, \tcb{\overline{\mu}_n}) \label{eq:CM1} \\
&\underline{G}_n \leq G_n \leq \overline{G}_n, \forall n \in \mathcal{N}, \hspace{4.5cm} & (\tcb{\underline{\nu}_n}, \tcb{\overline{\nu}_n}) \label{eq:CM2} \\
& q_{mn} \leq \kappa_{mn}, \forall m \in \Omega_n, m \neq n, \forall n \in \mathcal{N}, \hspace{2.6cm}  & (\tcb{\xi_{nm}}) \label{eq:CM3} \\
& q_{mn} \leq - q_{nm}, \forall m \in \Omega_n, m > n, \forall n \in \mathcal{N},& \hspace{2.4cm} (\tcb{{\zeta}_{nm}}) \label{eq:CM4} \\
& D_n = G_n + \Delta G_n + Q_n, \forall n \in \mathcal{N}. \hspace{3.2cm} & (\tcb{\lambda_n}) \label{eq:CM5}
\end{align}
\end{subequations}
\begin{remark}
The constraint \eqref{eq:CM4} is indexed by $m>n $ so that the constraint is considered only once.
\end{remark}

Dual variables are denoted in blue font between brackets at the right of the corresponding constraints. Some of the dual variables can be interpreted as shadow prices, with classical interpretations in the energy economics literature. In the remainder, $\xi_{nm}$ will be interpreted as the shadow price (congestion price) associated with capacity trading flow constraint \eqref{eq:CM3} between nodes $n$ and $m$; $\zeta_{nm}$ will be understood as the bilateral trade price offered by $n$ to $m$ associated with the trading reciprocity constraint \eqref{eq:CM4}; while $\lambda_n$ is the nodal price associated with the supply and demand balancing constraint in node $n$ \eqref{eq:CM5}, as discussed in Subsection~\ref{subsec:local_balancing}.  

The Social Welfare function is concave as the sum of concave functions defined on a convex feasibility set. Indeed, the feasibility set is obtained as Cartesian product of convex sets. We can compute the Lagrangian function associated with the standard constrained optimization problem of social welfare maximization under constraints (\ref{eq:CM1})-(\ref{eq:CM5}):
\begin{equation}
\begin{split}
&\mathcal{L}(\bd{D},\bd{G},\bd{Q},\bd{\mu},\bd{\nu},\bd{\xi},\bd{\zeta},\bd{\lambda}) = -SW(\bd{D},\bd{G},\bd{q})+\sum_{n\in\mathcal{N}}\underline{\mu}_n(\underline{D}_n-D_n) \\
+&\sum_{n\in\mathcal{N}}\overline{\mu}_n(D_n-\overline{D}_n) +\sum_{n\in \mathcal{N}}\underline{\nu}_n(\underline{G}_n-G_n)+\sum_{n\in \mathcal{N}}\overline{\nu}_n(G_n-\overline{G}_n)  \\
+&\sum_{n\in\mathcal{N}}\sum_{m\in\Omega_n,m\neq n}\xi_{nm} (q_{mn}-\kappa_{mn}) +\sum_{n\in\mathcal{N}}\sum_{m \in \Omega_n, m > n}\zeta_{nm}(q_{mn}+q_{nm}) \\
+&\sum_{n\in\mathcal{N}}\lambda_n\Big(D_n-G_n-\Delta G_n-Q_n\Big). \label{eq:CMLagrangian}
\end{split}
\end{equation}

To determine the solution of the centralized market design optimization problem, we compute KKT conditions associated with Lagrangian function (\ref{eq:CMLagrangian}). Taking the derivative of the Lagrangian function (\ref{eq:CMLagrangian}) with respect to $D_n$, $G_n$, $q_{mn}$, for all $n$ in $\mathcal{N}$ and all $m\in\Omega_n, m\neq n$, the stationarity conditions write down as follows:
\begin{subequations}
\begin{align}
\dfrac{\partial \L}{\partial D_n}=0 \Leftrightarrow \ &2 \tilde{a}_n(D_n-D_n^{\star})-\underline{\mu}_n+\overline{\mu}_n+\lambda_n = 0\, , \;\quad \forall n \in \mathcal{N}\, , \label{eq:KKT1} \\
\dfrac{\partial \L}{\partial G_n}=0 \Leftrightarrow \ &a_n G_n + b_n - \underline{\nu}_n+\overline{\nu}_n-\lambda_n = 0\, , \,\quad \quad\quad  \forall n \in \mathcal{N}\, , \label{eq:KKT2} \\
\dfrac{\partial \L}{\partial q\mn}=0 \Leftrightarrow \ & c_{nm}+\xi_{nm}+\zeta_{nm}-\lambda_n = 0,\quad\qquad \qquad \forall m \in \Omega_n, m \neq n, \forall n \in \mathcal{N}\, , \label{eq:KKT3}
\end{align}
\end{subequations}
where, for $m < n$, $\zeta_{nm}$ is defined as equal to $ \zeta_{mn}$.

From Equation~(\ref{eq:KKT3}), we infer that the nodal price at $n$ can be expressed analytically as the sum of the node product differentiation prices regarding the other prosumers in her neighborhood, the congestion constraint dual variable from Equation~(\ref{eq:CM3}) and the bilateral trade prices:
\begin{equation}
\lambda_n = c_{nm} + \xi_{nm} + \zeta_{nm} , \quad \forall m \in \Omega_n, m \neq n,\quad  \forall n \in \mathcal{N}\, . \label{eq:zetaopt}
\end{equation}

The complementarity constraints\footnote{A complementarity constraint enforces that two variables are complementary to each other, i.e., for two scalar variables $x,y$: $xy=0$, $x\geq 0$, $y \geq 0$. This condition is often expressed more compactly as: $0\leq x \perp y \geq 0$.} take the following form:
\begin{subequations}
\begin{align}
0 \leq \underline{\mu}_n &\perp D_n-\underline{D}_n \geq 0, \quad \forall n \in \mathcal{N}\, , \label{eq:cc1} \\
0 \leq \overline{\mu}_n &\perp \overline{D}_n-D_n  \geq 0, \quad \forall n \in \mathcal{N}\, , \label{eq:cc2} \\
0 \leq \underline{\nu}_n &\perp G_n-\underline{G}_n \geq 0, \quad \forall n \in \mathcal{N}\, , \label{eq:cc3} \\
0 \leq \overline{\nu}_n &\perp \overline{G}_n-G_n\geq 0, \quad \forall n \in \mathcal{N}\, , \label{eq:cc4} \\
0 \leq \xi_{nm} &\perp \kappa_{mn} - q_{mn}\geq 0, \quad \forall m \in \Omega_n, m \neq n, \forall n \in \mathcal{N}\, , \label{eq:cc5} \\
0 \leq \zeta_{nm} &\perp -q_{mn}-q_{nm}\geq 0, \quad \forall m \in \Omega_n, m>n, \forall n \in \mc{N}\, . \label{eq:cc6}
\end{align}
\end{subequations}

From Equation~(\ref{eq:KKT3}), we infer, for any couple of nodes $n \in \mathcal{N}, m \in \Omega_n, m > n$, that: 
\begin{equation} \label{eq:zetanm} 
\zeta_{nm} = \lambda_n - c_{nm} - \xi_{nm} = \lambda_m - c_{mn} - \xi_{mn}\, ,
\end{equation}
Subtracting those two last members in (\ref{eq:zetanm}), we infer that:
\begin{equation}
c_{nm} - c_{mn} + \xi_{nm} - \xi_{mn} = \lambda_n - \lambda_m, \forall m \in \Omega_n, m \neq n, \forall n \in \mathcal{N}\, .\label{eq:diffc}
\end{equation}

From Equations~(\ref{eq:KKT1}) and (\ref{eq:KKT2}), we infer that, at the optimum, for each node $n$:
\begin{align}
D_n =& D_n^{\star} - \frac{1}{2\tilde{a}_n}\Big( \lambda_n + (\overline{\mu}_n - \underline{\mu}_n)\Big)\, ,  \label{eq:Dopt}\\
G_n =& -\frac{b_n}{a_n}+ \frac{1}{a_n}\Big(\lambda_n - (\overline{\nu}_n-\underline{\nu}_n)\Big)\, . \label{eq:Gopt}
\end{align}

Substituting Equations~(\ref{eq:Dopt}) and (\ref{eq:Gopt}) in the local demand and supply balance Equation~(\ref{eq:CM5}), we infer that the net import at node $n$ can be expressed as a linear function of the nodal price:
\begin{equation}\label{eq:Qn}
Q_n = \Big( D_n^{\star}-\frac{1}{2 \tilde{a}_n}(\overline{\mu}_n-\underline{\mu}_n)+\frac{b_n}{a_n}+\frac{1}{a_n}(\overline{\nu}_n-\underline{\nu}_n)\Big) - \left(\frac{1}{2\tilde{a}_n}+\frac{1}{a_n}\right)\lambda_n - \Delta G_n\, . 
\end{equation}

The results are summarized in the following proposition.

\begin{proposition}\label{prop:optimum}
At the optimum, the demand, flexibility activation and net import at each node $n$ can be expressed as linear functions of the nodal price at that node, given by Equations \eqref{eq:Dopt}, \eqref{eq:Gopt}, and \eqref{eq:Qn}.
\end{proposition}

The total sum of the net imports at all nodes should be negative or null, i.e., $\sum_{n\in\mc{N}}Q_n\leq 0$. From the supply-demand balancing \eqref{eq:CM5}, this is equivalent to $\sum_{n\in\mc{N}}(D_n-G_n)\leq\sum_{n\in\mc{N}}\Delta G_n$. A strict inequality would lead to a situation where a part of the energy generation is wasted. That is not acceptable. To avoid that situation, the RES-based generation should be limited and the demand capacities large enough. Note that the sizing of the prosumers' capacities and RES-based generation possible clipping strategies are out of the scope of this work. This result is formalized in the proposition below.   

\begin{proposition}\label{prop:no_waste}
A necessary condition for no energy waste is that there is at least one prosumer $n$ in $\mc{N}$ whose capacities and RES-based generation are such that $\overline{D}_n-\underline{G}_n\geq\Delta G_n$.
\end{proposition}

\begin{proof}
By combining \eqref{eq:CM1} and \eqref{eq:CM2}, we obtain $\underline{D}_n-\overline{G}_n \leq D_n - G_n \leq \overline{D}_n-\underline{G}_n$. Subtracting $\Delta G_n$ in each part of the inequalities and applying \eqref{eq:CM5}, we get $\underline{D}_n-\overline{G}_n-\Delta G_n \leq Q_n \leq \overline{D}_n-\underline{G}_n-\Delta G_n$. Then, $\overline{D}_n-\underline{G}_n-\Delta G_n<0$ implies that $Q_n<0$, i.e., there are more exports than imports from $n$. If $\overline{D}_n-\underline{G}_n-\Delta G_n<0$, for all $n \in \mc{N}$ then, $\sum_{n\in\mc{N}}Q_n<0$. No energy waste is equivalent to $\sum_{n\in\mc{N}}Q_n=0$. For this equality to hold, it is necessary that there exists at least one prosumer $n$ in $\mc{N}$ such that $\overline{D}_n-\underline{G}_n\geq \Delta G_n$.
\end{proof}

In practice, this means that the prosumer should size their capacities such that the difference between their upper-bound on demand capacity and lower-bound on flexibility activation capacity is larger than their RES-based generation. However, the previous proposition is a necessary condition.

The following proposition gives a sufficient condition on the locational marginal prices $(\lambda_n)_n$ for having  no waste at optimality:

\begin{proposition}\label{prop:noWasteMarginalPrices}
At the optimum, if for any prosumer node $m$, for any  node $n_0$ such that there exists a non congested path  $(n_0,n_1,\dots,n_p=m)$ from $n_0$ to $m$ such that $\lambda_m > c_{n_0,m_0}+  \sum_{k=0}^{p-1} c_{n_k}-c_{n_{k+1}}$, where $m_0 \in \Omega_{n_0}$, then there is no energy waste at $n_0$ in the trade with $m_0$ (that is: $q_{n_0,m_0}+q_{m_0,n_0}=0$). 
In particular:
\begin{itemize}
    \item if users have symmetric preferences $c\nm=c\mn$, there is no congestion and there exists $m$ such that $\lambda_{m} > c_{n_0,m_0}$, then there is no energy waste at $n_0$ in the trade with $m_0$ ;
    \item for $m= n_0$, if $\lambda_{n_0} > c_{n_0,m_0}$, then there is no waste at $n_0$ in the trade with $m_0$, which can be directly inferred by the complementarity condition  \eqref{eq:cc6} and \eqref{eq:KKT3}.
\end{itemize}
\end{proposition}
\begin{proof}
Suppose on the contrary that there is some energy waste at $n_0$: there exists some $m_0$ such that $q_{n_0,m_0}+q_{m_0,n_0}<0$ and $q_{m_0,n_0}<0$ (i.e. $n_0$ rejects 
energy). In the case where $G_m > \ul{G}_m$,  Consider the infinitesimal transformation to the trades and production:
\begin{equation} \label{eq:changeTradesWaste} 
\begin{split}
    & q_{n_i,n_{i+1}} \leftarrow q_{n_i,n_{i+1}}+\varepsilon  , \quad \quad  q_{n_{i+1},n_i} \leftarrow q_{n_{i+1},n_i}-\varepsilon, \ \ \forall i \in \{0,\dots, p-1 \},  \\
     & \ q_{m_0,n_0} \leftarrow q_{m_0,n_0}+ \varepsilon \ , \ \ \qquad G_m \leftarrow G_m - \varepsilon \ .
     \end{split}
\end{equation}
Then, for $\varepsilon$ small enough, all constraints are still satisfied and the variations in $\SW$ has the same sign as:
\begin{equation*}
    \lambda_m -c_{n_0,m_0} + \sum_{i=0}^{p-1} ( c_{n_i,n_{i+1}}-c_{n_{i+1},n_i}) >0 \ .
\end{equation*}
Hence, we can strictly increase $\SW$, which contradicts the optimality. In the case where $G_m = \ul{G}_m$, then we necessarily have $D_m <\ol{D}_m$ (otherwise $\lambda_n = -2 \tilde{a}_n (\ol{D}_n-D^*_n)-\ol{\mu}_n <0 $ which is impossible from \eqref{eq:KKT3}), and we can strictly increase $D_m$ instead of decreasing $G_m$  in \eqref{eq:changeTradesWaste}, leading to the same contradiction.
\end{proof}
\begin{remark}
From the previous proposition, we see that even if there is no excess in the renewable production, i.e. $\sum_n \Delta G_n < \sum_n D^*_n$, we can still have some energy waste if the trades preference prices $(c\nm)$ are large enough.
\end{remark}

Hence, assuming no energy waste, the total sum of the net imports in all nodes should vanish, which implies the following relation: 
\begin{align}
&\sum_{n\in\mathcal{N}}Q_n = 0 \nonumber \\
&\Leftrightarrow \sum_{n \in \mathcal{N}}\left(\frac{1}{2\tilde{a}_n}+\frac{1}{a_n}\right)\lambda_n  = \sum_{n \in \mathcal{N}}\Big(D_n^{\star}-\frac{1}{2\tilde{a}_n}(\overline{\mu}_n-\underline{\mu}_n)+\frac{b_n}{a_n}+\frac{1}{a_n}(\overline{\nu}_n-\underline{\nu}_n)-\Delta G_n\Big)\, , \label{eq:lambda}
\end{align}
using Equation~(\ref{eq:Qn}). 

From Equation~(\ref{eq:diffc}), we infer that the nodal price at node $n$ is a linear function of the nodal price at the root node, product differentiation and congestion prices with all the other nodes in $\mathcal{N}$:
\begin{equation}\label{eq:lambdan}
\lambda_n = c_{n0} - c_{0n} + \xi_{n0}-\xi_{0n} + \lambda_0,\quad \forall n \in \Omega_0\, .
\end{equation}

Substituting Equation~(\ref{eq:lambdan}) in Equation~(\ref{eq:lambda}), we infer the closed form expression of the nodal price at the root node:
\begin{align}
\lambda_0 \sum_{n \in \mathcal{N}}\left(\frac{1}{2\tilde{a}_n}+\frac{1}{a_n}\right) =& \sum_{n \in \mathcal{N}}\Big( D_n^{\star}-\frac{1}{2\tilde{a}_n}(\overline{\mu}_n - \underline{\mu}_n)+\frac{b_n}{a_n}+\frac{1}{a_n}(\overline{\nu}_n-\underline{\nu}_n)-\Delta G_n\Big) \nonumber \\
-&\sum_{n \in \Omega_0}\left(\frac{1}{2\tilde{a}_n}+\frac{1}{a_n}\right)\big( c_{n0}-c_{0n}+\xi_{n0}-\xi_{0n}\big)\ .\label{eq:lambda0}
\end{align}

From Equations~(\ref{eq:lambdan}) and (\ref{eq:lambda0}), assuming that $(c_{n0})_{n}$, $(c_{0n})_{n}$, $(\xi_{n0})_{n}$, $(\xi_{0n})_{n}$ are known, the MO can iteratively compute all the $(\lambda_n)_{n \in \mathcal{N}}$. Note that $\underline{\bm{\mu}},\overline{\bm{\mu}}$ and $\underline{\bm{\nu}},\overline{\bm{\nu}}$ are determined by the MO when optimizing $\bm{D}$ and $\bm{G}$. Once computed by the MO, the nodal prices are announced to all the agents $n \in \mathcal{N}$. Then, to determine the optimal bilateral trading prices, each agent $n$ has to refer to Equation~(\ref{eq:zetanm}), which gives the bilateral trading prices as linear functions of the nodal price and congestion price. The results are summarized in the following proposition:

\begin{proposition}\label{prop:optimumPrice}
Assuming no energy waste and knowing $(c_{n0})_{n}$, $(c_{0n})_{n}$, $(\xi_{n0})_{n}$, $(\xi_{0n})_{n}$, the MO computes the nodal price at the root node by Equation~\eqref{eq:lambda0}. The nodal prices in all the other nodes of the distribution network can be inferred from $\lambda_0$ according to Equation~\eqref{eq:lambdan}. Then, for each node $n \in \mathcal{N}$, bilateral trading prices can be computed for any node $m \in \Omega_n, n \neq m$ by Equation~\eqref{eq:zetanm} provided congestion price $(\xi_{nm})_{m>n,m\in\Omega_n}$ is known\footnote{Two assumptions can be made on the determination of the congestion prices: first, they are determined exogenously while checking the complementarity constraint \eqref{eq:cc5}; second, they are determined through a market for (distribution) capacity line transmission. This second assumption enables the MO to complete the market. It will be discussed later in the paper.}.
\end{proposition}

If all agents reveal their product differentiation prices $(c_{n0})_n$ to the MO and all the congestion prices $(\xi_{n0})_n$, $(\xi_{0n})_n$ in the lines involving the root node are known (or rationally anticipated), then the MO can compute all the nodal prices $(\lambda_n)_{n\in\mc{N}}$ from $\lambda_0$.
\bigskip

We now want to make the link between the market and the state of the distribution grid. In the following proposition, we show that the distribution grid lines become congested if there are ``cycles'' in the preferences as explained below.
\newcommand{\tC}{\tilde{C}}
\newcommand{\tq}{\tilde{q}}
\begin{proposition} \label{prop:fullCapaTrades}
Suppose that the matrix $\tC := (c\nm-c\mn)\nm $ has a strictly negative cycle of length $k>2$, i.e. there is a sequence of distinct indices $(n_i)_{1\leq i\leq k}$ such that $   \sum_{1\leq i \leq k} \tC_{n_i,n_{i+1}} < 0 $, where $n_{k+1}:=n_1$.
Then, at an optimal centralized solution, there is a trade opposed to the cycle made at full capacity, i.e. there exists $i\in \{1,\dots,k\}$ such that $q_{n_{i+1},n_i} = \kappa_{n_{i+1},n_i}$.

Symmetrically, if there is a strictly positive cycle  $(n_i)_{1\leq i\leq k}$ such that $   \sum_{1\leq i \leq k} \tC_{n_i,n_{i+1}} > 0 $, then at an optimal centralized solution, there is a trade in the direction of the cycle made at full capacity, i.e. there exists $i\in \{1,\dots,k\}$ such that $q_{n_i,n_{i+1}} = \kappa_{n_i,n_{i+1}} $.
\end{proposition}

\begin{proof}[Proof of \Cref{prop:fullCapaTrades}]
We prove the first part of the proposition as the second is symmetric.

Consider the trades $(q\nm)\nm$ at an optimal solution and suppose on the contrary that there is $\epsilon >0$ such that, for each $i\in \{1,\dots,k\}$, we have $q_{n_{i+1},n_i} \leq \kappa_{n_{i+1},n_i} - \epsilon$.

Then consider the same solution with trades $(\tq\nm)\nm$ defined as follows: for each  $i\in \{1,\dots,k\}$, let $\tq_{n_{i+1},n_i}:= q_{n_{i+1},n_i}+ \epsilon$ and $\tq_{n_i,n_{i+1}}:= q_{n_i,n_{i+1}}- \epsilon$, while $\tq\nm=q\nm$ otherwise. 
Then all constraints are still feasible because, for each $i$, $\sum_{m \neq n_i} q_{m,n_i} =Q_n - \epsilon + \epsilon =Q_n$. Besides, by definition of $\tq$, we still have $\tq\mn=-\tq\nm$ for any $m >n$. Moreover, if we denote by $\SW$ the social welfare of the previous solution $(q\nm)\nm$, the social welfare of this new solution is:
\begin{align*}
\widetilde{\SW} & =\SW + \sum_n \sum_{m \neq n} c\nm (q\mn-\tq\mn) \\
&= \SW + \sum_{1\leq i \leq k}  \Big( c_{n_i,n_{i+1}} (q_{n_{i+1},n_i}-\tq_{n_{i+1},n_i}) +c_{n_i,n_{i-1}} (q_{n_{i-1},n_i}-\tq_{n_{i-1},n_i})  \Big) \\
&= \SW + \sum_{1\leq i \leq k} \epsilon \left( c_{n_i,n_{i-1}} - c_{n_i,n_{i+1}} \right) =  \SW  - \epsilon \sum_{1\leq i \leq k}  \tC_{n_i,n_{i+1}}  > \SW \ ,
\end{align*}
which contradicts the fact that $\SW$ is maximal.
\end{proof}

\begin{remark} The  property  stated  by  \Cref{prop:fullCapaTrades}  shows  that  the  lines  become  congested  if  there  is a strictly positive or negative cycle in the matrix $\tC$. In practice, a central MO should  try  to  avoid such an outcome, since the  congested  lines  are  unavailable  in  case  of unplanned real  need (outages, peak demand).  The  existence  of  a positive  cycle  in $\tC$  means  that  there  is  an  ``arbitrage''  opportunity  in  the  network. In other words, one  can  strictly increase the social welfare by doing an exchange of power quantities.
We can make the assumption that this kind of opportunities do not exist in practice, since they should vanish quickly in a liquid market.

From the  point  of  view  from  mechanism  design,  we  might  also  prevent  this  kind  of  cycling  behavior
by  adding  a  transaction  fee  (e.g.
 $\tau \times | q\mn |$ with $\tau >0 $)  on  the  trades,  regardless  they  are  positive  or negative.
 \end{remark}
 \Cref{subsec:3nodeEx} shows an example where there is a cycling trade that is purely due to arbitrage opportunities because of the preferences.

\section{Peer-to-Peer Market Design}\label{sec:p2p_md}

The centralized market design is used, in this section, as a benchmark against which we test the performance of a fully distributed approach relying on peer-to-peer energy trading. We first start by defining in Subsection~\ref{sec:def} the solution concepts that we will use to analyze the outcome of the fully distributed market design. Then, various results are introduced to characterize the relations between these sets of solutions. Congestion issues and performance measures are discussed in Subsection~\ref{sec:perf}. 

\newcommand{\ta}{\tilde{a} } 
\newcommand{\nmi}{_{nm} } 
\newcommand{\mni}{_{mn} } 
\newcommand{\omu}{\overline{\mu}}
\newcommand{\umu}{\underline{\mu}}

\subsection{General Nash Equilibrium and Variational Equilibrium}\label{sec:def}
In the peer-to-peer setting, each agent $n \in \mathcal{N}$ determines, by herself, her demand, flexibility activation and bilateral trades with other agents in her local energy community under constraints on demand, flexibility activation and transmission capacity so as to maximize her utility. 
A trade between two agents in a local energy community supposes that these two have decided on a certain quantity to be sent from one side and received by the other side. Therefore, there must be an ``agreement'' or trade constraint between each pair of agents in a local community, which couples their respective decisions. As a result, although the utility of a prosumer depends only on her own decisions, some of these decisions, such as the quantity she agrees to trade with all the other prosumers in her neighborhood, have an impact on the set of feasible actions of her neighbors. In the same way, her feasible actions are determined by the actions of her neighbors. 

Formally, each agent in node $n \in \mathcal{N}$ solves the following optimization problem:
\begin{subequations}
\label{problemUsern}
\begin{align}
\max_{D_n, G_n, (q_{mn})_{m \in \Omega_n, m\neq n}} \hspace{1cm} & \Pi_n\Big(D_n, G_n, \bm{q}_n\Big), \label{eq:p2p_objectif} \\
s.t. \hspace{1cm} & \underline{D}_n \leq D_n \leq \overline{D}_n, \hspace{3.9cm} & (\tcb{\underline{\mu}_n}, \tcb{\overline{\mu}_n}) \label{cd:capaD} \\
& \underline{G}_n \leq G_n \leq \overline{G}_n, \hspace{4.cm} & (\tcb{\underline{\nu}_n}, \tcb{\overline{\nu}_n}) \label{cd:capaG} \\
 & q_{mn} \leq \kappa_{mn}, \forall m \in \Omega_n, m \neq n, \hspace{1.7cm} & (\tcb{\xi_{nm}}) \label{cd:capacity}\\
& q_{mn} \leq - q_{nm}, \forall m \in \Omega_n, m \neq n, \hspace{1.7cm} & (\tcb{\zeta_{nm}}) \label{cd:transaction}\\
& D_{n} = G_{n} + \Delta G_n + Q_{n} , \hspace{2.4cm} & (\tcb{\lambda_n}) \label{cd:game_D}
\end{align}
\end{subequations}
where $\bm{q}_n= (q_{mn})_{m\in\Omega_n}$ are the trading decisions of agent $n$.

Hence, the peer-to-peer setting leads to $N$ optimization problems, one for each agent $n \in \mathcal{N}$,  with \emph{individual constraints} on demand 
\eqref{cd:capaD}, flexibility activation \eqref{cd:capaG}, trade capacity \eqref{cd:capacity}, supply and demand balancing \eqref{cd:game_D}; as well as \emph{coupling constraints} \eqref{cd:transaction} that ensure the reciprocity of the trades. 

The Lagrangian function associated with optimization problem \eqref{eq:p2p_objectif} under constraints \eqref{cd:capaD}-\eqref{cd:game_D}, writes down as follows:
\begin{align*}
\mathcal{L}_n&(D_n, G_n,\bd{q}_n,\bm{\mu}_n,\bm{\nu}_n,\bm{\xi}_{n},\bm{\zeta}_{n}, \lambda_n)\\
=& \tilde{a}_n(D_n  -D^*_n)^2-\tilde{b}_n+\tfrac{1}{2}a_n G^2_n+b_n G_n +d_n + \sum_{m \in \Omega_n, m \neq n}c_{nm}q_{mn} \\
& +\underline{\mu}_n (\underline{D}_n-D_n) +\overline{\mu}_n (D_n  - \overline{D}_n) + \underline{\nu}_n(\underline{G}_n-G_n)+\overline{\nu}_n(G_n-\overline{G}_n)+\sum_{m \in \Omega_n, m \neq n} \xi_{nm}(q_{mn}-\kappa_{mn})\\
& +\sum_{m \in \Omega_n, m \neq n}\zeta_{nm}(q_{mn}+q_{nm})
+\lambda_n (D_n - G_n - \Delta G_n - \sum_{m\in\Omega_n, m\neq n} q_{mn}).
\end{align*}

For each agent $n$, the first order stationarity conditions are the same as \eqref{eq:KKT1}-\eqref{eq:KKT3}, and  the complementarity constraints are the same as \eqref{eq:cc1}-\eqref{eq:cc6}, except that \eqref{eq:cc6} is indexed by all $(m,n)$ with $m \neq n$ and that $\zeta\nm$ is not necessarily equal to $\zeta\mn$. Let this condition system be denoted by $K\!K\!T_n$ for each $n\in\mathcal{N}$. 

As the problem given by \eqref{problemUsern} is convex, $K\!K\!T_n$ are necessary and sufficient conditions for a vector $(D_n,G_n,\bm{q}_n)$ to be an optimal solution of \eqref{problemUsern}.

\begin{remark} \label{rm:gameWithoutGamingObjectives}
In Equation~\eqref{problemUsern}, $\Pi_n$ depends on the variables of player $n$ only, and not on the variables of the other players. A consequence is that the social welfare function is decomposable: $SW(D,G,\bm{q}) = \sum_n \Pi_n\big(D_n, G_n, \bm{q}_n\big)$. Therefore, without the existence of the coupling transaction constraint \eqref{cd:transaction}, the minimization of SW is equivalent to the minimization of each individual objective function $\Pi_n$. 
We will see that this equivalence between social optimizer and equilibria also happens for the so-called Variational Equilibria. 
\end{remark}

A common adopted equilibrium notion that generalizes Nash Equilibria in the presence of coupling constraints is the notion of Generalized Nash Equilibrium (GNE)  \cite{harker1991}

\begin{definition}[Generalized Nash Equilibrium \cite{facchinei2007}] \label{def:GNE}
A Generalized Nash Equilibrium of the game defined by the maximization problems \eqref{problemUsern} with coupling constraints, is a vector $(D_n,G_n,\bm{q}_n)_n$ that solves the maximization problems \eqref{problemUsern} or, equivalently, a vector $(D_n,G_n,\bm{q}_n)_n$ such that $(D_n,G_n,\bm{q}_n)$ solves the system $K\!K\!T_n$ for each $n$.
\end{definition}

The constraint \eqref{cd:transaction}, $q_{mn}\leq -q_{nm}$, written both in the problem of $n$ and in that of $m \neq n$ leads to the same inequality, but is associated to the multiplier $\zeta_{nm}$ in the problem of $n$ and to $\zeta_{mn}$ in the problem of $m$. 
In this paper, we consider two scenarios for the allocation of the resources represented in this coupling constraints:
\begin{itemize}[leftmargin=0.5cm,itemindent=2cm]
\item[\textbf{Scenario (i)}]  A market allocates the resources associated with \eqref{cd:transaction} through a single price system, therefore leading to the determination of one price for each constraint: $\zeta_{nm} = \zeta_{mn}$.

\item[\textbf{Scenario (ii)}] There does not exist any market to determine the price system associated with \eqref{cd:transaction}. Hence, two prosumers $n$, $m$ might attribute different evaluations of the same transaction $q_{mn}\leq -q_{nm}$ or, equivalently, the same dual variables to the trade constraint \eqref{cd:transaction} between $n$ and $m$. This \emph{can} lead to different prices $\zeta_{nm} \neq \zeta_{mn}$ for agents $n$ and $m$.  
\end{itemize}

The two scenarios have implications on the market organization. Let us discuss them one after another.

\smallskip

\textbf{Scenario (i)} corresponds to a complete market, where the common resources are shared in an efficient way. It suggests that all constraints are traded at a single price, which reflects the common valuation of each product from all agents.  The associated solution concept is that of Variational Equilibrium \cite{harker1991}, a refinement of Generalized Nash Equilibrium, where we ask for more symmetry: the Lagrangian multipliers associated to a constraint shared by several players have to be equal from one player to another. Note that a natural way to complete the market would be to introduce a market for (distribution) capacity line transmission, enabling the determination of congestion prices $(\xi_{nm})_{n,m}$. A similar idea was proposed by Oggioni et al. in \cite{oggioni} at the transmission level for a subproblem of market coupling. 

\begin{definition}[Variational Equilibrium \cite{facchinei2007}]
A Variational Equilibrium  of the game defined by \eqref{problemUsern} is a solution $(D_n,G_n,\bm{q}_n)_n$ that solves the maximization problems \eqref{problemUsern} or, equivalently, a vector $(D_n,G_n,\bm{q}_n)_n$ such that $(D_n,G_n,\bm{q}_n)$ solves the system $K\!K\!T_n$ for each $n$ and, in addition, such that the Lagrangian multipliers associated to the coupling constraints \eqref{cd:transaction} are equal, i.e.:
\begin{equation}
\zeta_{nm}=\zeta_{mn}, \ \forall n \in \mc{N}, \forall m \in \Omega_n, m \neq n \ .
\end{equation}
\end{definition}

 The term ``variational'' refers to the variational inequality problem associated to such an equilibrium: indeed, if we define the set of admissible solutions as:
 \begin{equation}
 \mc{R} \eqd \{ \xx= (D_n,G_n,\bm{q}_n)_n \ | \eqref{cd:capaD}-\eqref{cd:game_D} \text{ hold for each } n \in \mc{N}  \} \ . 
 \end{equation}
 then $\hat{\xx} \in \mc{R}$ is a Variational Equilibrium if, and only if, it is a solution of (cf. \cite{facchinei2007}):
 \begin{equation}
 \left\langle \sum_n \nabla  \Pi_n (\hat{\xx}_n), \ \xx - \hat{\xx}  \right\rangle \leq 0 , \ \forall \xx \in \mc{R} \ . \label{eq:VIofVE}
 \end{equation}
 A remarkable fact is that  Variational Equilibria exist under mild conditions \cite{harker1991,rosen1965}, even if the additional equality conditions on the multipliers seem restrictive.

We can observe, following Remark~\ref{rm:gameWithoutGamingObjectives}, that Variational Equilibria are defined by exactly the same KKT system than the social welfare maximizer (or equivalently as the solution of the same variational inequality \eqref{eq:VIofVE}). Therefore, we obtain the following result:

\begin{proposition}\label{prop:GNE_opt}
The set of Variational Equilibria (such that $\zeta_{nm}=\zeta_{mn}$ for all $n \in \mathcal{N}$ and all $m\neq n \in \Omega_n$) coincides with the set of social welfare optima.
\end{proposition}

\medskip

\textbf{Scenario (ii)} corresponds to the case of partial price coordination or a completely missing market for some products. Agents with different willingness to pay for a certain resource face a price gap due to the lack of arbitrage opportunities that prevent price convergence. This imperfect coordination among agents relates to the notion of Generalized Nash Equilibrium (GNE), where nothing prevents the multipliers $\zeta_{nm}$ and $\zeta_{mn}$ to be different.

\begin{remark}
A particular class of GNE is called restricted GNE \cite{fukushima}. It assumes that the dual variables of the shared constraint \eqref{cd:transaction} belongs to a non empty cone of $\mathbb{R}^{N(N-1)}$. 

A particular class of restricted GNE is called normalized equilibrium, introduced by Rosen \cite{rosen1965}. There, the dual variables of the shared constraint \eqref{cd:transaction} are equal up to a constant endogenously given factor
$r_n$ that depends on prosumer $n$, but not on constraints. Mathematically, it means  $r_n \zeta_{nm} = r_m \zeta_{mn}$, for all $n \in \mathcal{N}$ and all $m \in \Omega_n, m \neq n$.
\end{remark}

From $K\!K\!T_n$, we see that, as in the centralized case, $\lambda_n = \zeta_{nm} + c_{nm} + \xi_{nm}$, i.e., the per-unit nodal price at $n$ is the sum of the transaction price, the preference price and the congestion price, all for getting one unit from $m$ to $n$, for each neighbor of $m$. 

Besides,
\begin{equation}\label{eq:zeta_p2p}
\zeta_{nm} = \lambda_n - c_{nm} - \xi_{nm}, \forall m \in \Omega_n, m \neq n\ ,
\end{equation}
which gives the transaction price for agent $n$ or, in other words, her evaluation of the trade $q_{mn}$. 

\bigskip

In order to derive some results on GNE and simplify notations, let us introduce the coefficient $r_n$ as:
\begin{equation}
    \zeta_{0n} r_n = \zeta_{n0}, \quad \forall n \in \mathcal{N} \ .
\end{equation}

\begin{remark}
We interpret this situation as one where there is an imperfect market for determining the bilateral trade prices obtained as dual variables of the shared constraint \eqref{cd:transaction}. Between any couple of prosumer nodes, bilateral trade prices do tend to equalize (i.e., $r_n$ is close to $1$ for any $n \in \Omega_0$ --- meaning that the GNE approaches the Variational Equilibrium), but there remains a gap due to insufficient liquidity or differences in the price bids for the asked quantity \cite{oggioni}. To some extent, $r_n$ can be interpreted as a measure of the efficiency loss introduced by the GNE in comparison with the Variational Equilibrium.  \end{remark}

Using Equation \eqref{eq:zeta_p2p} for the node 0 and an arbitrary node $n\in  \Omega_0$ and for an arbitrary node $n \in \Omega_0$ and the node 0, and summing up both relations, we get:
\begin{align} \label{eq:lambdan_GNE}
&  \lambda_n = r_n \lambda_0+\big( c_{n0}-r_n c_{0n} \big) + \big( \xi_{n0}  - r_n \xi_{0n}\big),  \ \forall n \in \Omega_0.
\end{align}

Similarly to the centralized market design case, since the total sum of the net imports in all nodes should vanish under no RES-based generation waste, i.e., $\sum_{n\in\mathcal{N}}Q_n=0$, we infer the closed form expression of the nodal price at the root node, similar to the centralized case:
\begin{align}
 \lambda_0  \sum_{n \in \mathcal{N}}\left(\frac{1}{2\tilde{a}_n}+\frac{1}{a_n}\right)r_n  =&  \sum_{n \in \mathcal{N}} \Big(  D_n^{\star} -\frac{1}{2\tilde{a}_n}(\overline{\mu}_n-\underline{\mu}_n)+\frac{b_n}{a_n}+\frac{1}{a_n}(\overline{\nu}_n-\underline{\nu}_n) 
-\Delta G_n \Big) \nonumber \\
-& \sum_{n\in\Omega_0}\left(\frac{1}{2\tilde{a}_n}+\frac{1}{a_n}\right) \Big[ (c_{n0}-c_{0n} r_n )+\big(  \xi_{n0}  - r_n \xi_{0n}\big)\Big]. \label{eq:lambda0_GNE}
\end{align}

We introduce $\text{SOL}^{\text{GNEP}}$ as the set of GNE solutions of the peer-to-peer non-cooperative game.  

As opposed to the Variational Equilibrium, GNEs are not unique in general. It is relevant to study how efficient those different outcomes can be in comparison  to the Variational Equilibrium outcome (where the bilateral trades would be settled down by a MO).

\newcommand{\Pgne}{\P^{\GNE}_{\bm{\omega}}}
To that purpose, we apply the parameterized variational inequality approach \cite{nabetani2009,oggioni}, which enables to evaluate the set of GNEs. In our specific case, this leads to  the following  optimization problem $\Pgne$, parameterized by the coefficients $\omega\nm>0$ corresponding to an additional value for user $n$ for its trading constraint with $m$:
\begin{subequations}
\label{eq:pbAuxiliaryGNE}
\begin{align} 
 \Pgne \quad \quad \max_{\bm{D},\bm{G},\bm{q}} & \sum_{n\in\mathcal{N}} \left[ \Pi_n(D_n,G_n,\bm{q}_n) -  \sum_{m\in\Omega_n,m\neq n} \omega\nm q\mn \right], \label{eq:paraVI1} \\ 
s.t. \ \ & \underline{D}_n \leq D_n \leq \overline{D}_n, \forall n \in \mathcal{N},   & (\tcb{\underline{\mu}_n},\tcb{\overline{\mu}_n})  \label{eq:paraVI2} \\
&  \underline{G}_n \leq G_n \leq \overline{G}_n, \forall n \in \mathcal{N},  &(\tcb{\underline{\nu}_n},\tcb{\overline{\nu}_n})  \label{eq:paraVI3} \\
& q_{nm}\leq \kappa_{nm}  & (\tcb{\xi_{nm}}) \label{eq:paraVI4} \\
&  q_{nm}+q_{mn} \leq 0, \forall m \in\Omega_n, m > n, \forall n \in \mathcal{N}.  & (\tcb{\zeta_{nm}}) \label{eq:paraVI5} \\
&  D_n = G_n +\Delta G_n + Q_n, \forall n \in \mathcal{N},  & (\tcb{\lambda_n}) \label{eq:paraVI6} \ .
\end{align}
\end{subequations}
Indeed, from \cite[Cor 3.1]{nabetani2009} and \cite[Thm.~3.3]{nabetani2009}, we have the following results:
\begin{proposition} \label{prop:GNEparam}
\begin{enumerate}[label=(\roman*)]
    \item All GNEs can be found from problem \eqref{eq:pbAuxiliaryGNE}, that is:
    \begin{equation*}
\text{SOL}^{\text{GNEP}} \subset \bigcup_{(\omega\nm) \in \rr^{*N(N-1)}_{+}} \text{SOL}\left(\Pgne\right) \ ;
\end{equation*}
    \item reciprocally, if $(\bm{D},\bm{G},\bm{q}, \bm{\zeta})$ is a solution of $\Pgne$ (where  $\bm{\zeta}$ are multipliers associated to \eqref{eq:paraVI5}), then
    \begin{equation}
(\bm{D},\bm{G},\bm{q}, \bm{\zeta}) \text{ is a GNE} \ \Longleftrightarrow \ \omega\nm (q\nm+q\mn)=0, \  \forall n\neq m \ , 
\end{equation}
and in that case the multipliers associated to \eqref{cd:transaction} in the GNE problem are defined by $\hat{\zeta}\nm= \zeta\nm+ \omega\nm$ .
\end{enumerate}
\end{proposition}
\begin{proof}
For (i), writing the KKT conditions verified by a solution $(\bm{D},\bm{G},\bm{q})$ of the GNE problem \eqref{problemUsern} with Lagrangian multipliers $(\hat{\zeta}\nm)_{n\neq m}$, it is easy to verify that $(\bm{D},\bm{G},\bm{q})$ verifies the KKT conditions of \eqref{eq:pbAuxiliaryGNE} $\P^{\GNE}_{\bm{\hat{\zeta}}}$, where the parameters are taken to $\bm{\omega}:=\bm{\hat{\zeta}}$.

For (ii), we use the fact that problem \eqref{eq:pbAuxiliaryGNE} has linearly independent constraints, and apply \cite[Thm.~3.3]{nabetani2009} directly.
\end{proof}

\subsection{Dealing with Congestion}\label{sec:perf}

Let us first explicit the following fact on congested lines:
\begin{lemma}\label{lemma:congestion_unilateral}
For any couple of nodes $n \in \mc{N}, m \in \Omega_n, m \neq n$, such that $\kappa_{nm}>0$, $\kappa_{mn}>0$, $q_{nm}=\kappa_{nm}$ and $q_{mn}=\kappa_{mn}$ cannot hold simultaneously.
\end{lemma}
The proof is direct from the capacity and transaction constraints. Then, we obtain the following sufficient condition for a line to be saturated:
\begin{proposition}\label{prop:preference_asymmetry}
Suppose $\xi_{n0}=\xi_{0n}=0, \forall n \in \Omega_0$, i.e., there are large line capacities from and to node $0$, $c_{n0} = c_{m0}$, i.e., the nodes have the same preferences for node $0$, and the node $0$ has the same preferences for any node, i.e., $c_{0n}=c_{0m}$. For any couple of nodes $n \in \mc{N}, m \in \Omega_n, m\neq n$, asymmetric preferences (such as $c_{mn}>c_{nm}$ or $c_{mn}<c_{nm}$) imply that the node with the smaller preference for the other saturates the line. 
\end{proposition}

\begin{proof}
For any $n \in \mc{N}, m \in \Omega_n, m \neq n$, applying Equation~\eqref{eq:diffc} for the three couples of nodes: $(n,0)$, $(0,m)$, $(m,n)$, we obtain:
\begin{align*}
c_{n0}-c_{0n}+\xi_{n0}-\xi_{0n}=&\lambda_n-\lambda_0, \\
c_{0m}-c_{m0}+\xi_{0m}-\xi_{m0}=&\lambda_0-\lambda_m, \\
c_{mn}-c_{nm}+\xi_{mn}-\xi_{nm}=&\lambda_m-\lambda_n.
\end{align*}

Summing up the three equations, we get:
\begin{equation*}
\xi_{nm}-\xi_{mn} = (c_{0m}-c_{0n}) + (c_{n0} - c_{m0}) + (\xi_{n0} - \xi_{0n}) + (\xi_{0m} - \xi_{m0}) + (c_{mn} - c_{nm}). 
\end{equation*}

Under the assumptions of the proposition, the equation can be simplified to give:
\begin{equation*}
    \xi_{nm}-\xi_{mn} = c_{mn}-c_{nm}.
\end{equation*}

Then, two cases arise depending on the order of $(n,m)$ preferences:
\begin{itemize}
    \item[] \textbf{(i)} If $c_{mn}>c_{nm}$ (meaning that $m$ wants(buys) to sell to $n$ more(less) than $n$ wants to sell(buy) to $m$), $\xi_{nm}-\xi_{mn}>0$, which implies from Lemma~\ref{lemma:congestion_unilateral} that $\xi_{nm}>0$. Then, for the complementarity constraint \eqref{eq:cc5} to hold we need to have $q_{nm}=\kappa_{nm}$, i.e., $m$ saturates the line from $m$ to $n$;
    \item[] \textbf{(ii)} If $c_{nm}>c_{mn}$ (meaning that $n$ wants(buys) to sell to $m$ more(less) than $m$ wants to sell(buy) to $n$), $\xi_{nm}-\xi_{mn}<0$, which implies from Lemma~\ref{lemma:congestion_unilateral} that $\xi_{mn}>0$. Then, for the complementarity constraint \eqref{eq:cc5} to hold we need to have $q_{mn}=\kappa_{mn}$, i.e., $n$ saturates the line from $n$ to $m$.
\end{itemize}
\end{proof}

The following proposition gives a sufficient condition for the distribution grid lines become congested along a cycle, analog to \Cref{prop:fullCapaTrades}. The proof is similar and is omitted.
\begin{proposition} \label{prop:fullCapaTradesGame}
Suppose that there is a sequence of distinct indices $(n_i)_{1\leq i\leq k}$ such that $ \tC_{n_i,n_{i+1}}-\tC_{n_i,n_{i-1}} < 0 $ for all $i= 1, \ldots, k$, where $n_{k+1}:=n_1$.
Then, at an equilibrium, there is a trade opposed to the cycle made at full capacity i.e. there exists $i\in \{1,\dots,k\}$ such that $q_{n_{i+1},n_i} = \kappa_{n_{i+1},n_i} $.
\end{proposition}

\begin{remark}
Classically, the Price of Anarchy (PoA) is introduced as a performance measure to assess the performance of the peer-to-peer market design by comparison to the centralized market design. The PoA is defined as the ratio of the social welfare evaluated in the social welfare optimum to the social welfare evaluated in the worst GNE in the set $\textrm{SOL}^{\textrm{GNEP}}$. Formally, it is defined as follows:
\begin{equation}\label{eq:PoA}
PoA := \frac{\max_{\mathbf{D},\mathbf{G},\mathbf{q}} SW(\mathbf{D},\mathbf{G},\mathbf{q})}{\min_{\mathbf{D},\mathbf{G},\mathbf{q} \in \text{SOL}^{\text{GNEP}}} SW(\mathbf{D},\mathbf{G},\mathbf{q})}.
\end{equation}
From Proposition~\ref{prop:GNE_opt}, in a Variational Equilibrium, $PoA=1$, because a Variational Equilibrium coincides with the optimum of the centralized social welfare optimization problem. However, the GNE set might contain equilibria that do not coincide with the (social welfare) optimum solution of the centralized optimization problem.
\end{remark}

\section{Test Cases}\label{sec:test_cases}

\subsection{A Three Nodes Network with Arbitrage Opportunity}\label{ex:ex1}
\label{subsec:3nodeEx}
In this section, we first present a toy model with only three nodes indexed by $\{0,1,2\}$, as  illustrated in Figure~\ref{fig:scheme_3Nodes}. The root node $0$ has only conventional generation ($\Delta G_0=0$) with cost $(a_0,b_0)=(4,30)$ and $(\uG,\oG)=(0,10)$. Nodes $1$ and $2$ are prosumers with RES-based generators ($\oG_n=\uG_n=0$ and $\Delta G_n>0$ for $n\in \{1,2\}$). Each node is a consumer (with $(\ul{D},\ol{D})=(0,10)$) and  generator (RES or conventional), therefore producing energy that can be consumed locally to meet demand $D_n$ and exported to the other nodes to meet the unsatisfied  demand.

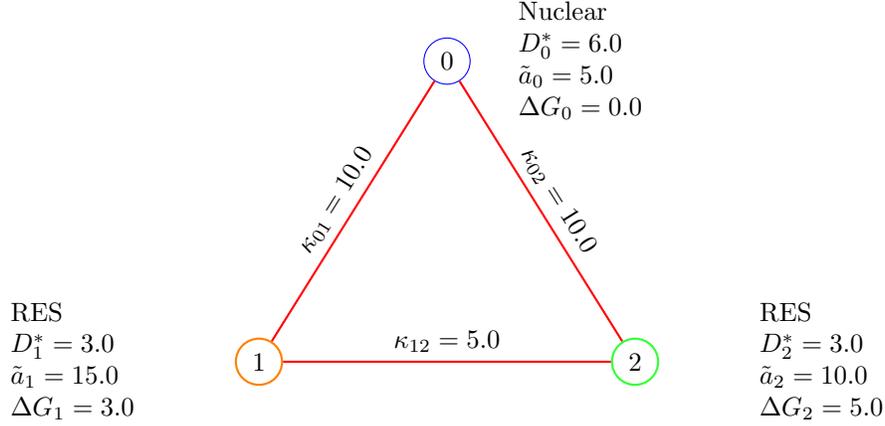
\begin{figure}[!ht]
\begin{center}
\begin{tikzpicture}
\node [draw,circle,blue,fill=white,text=black] (n0) at (0,4) {0};
\node [draw,thick,circle, orange,text=black] (n1) at (-2.5,0) {1};
\node [draw,thick,circle, green!80,text=black] (n2) at (2.5,0) {2};
\node [ right= 0.3cm of n0] (n0p) {\begin{tabular}{l} 
Nuclear\\ 
 $D^*_0= 6.0$ \\ 
$ \tilde{a}_0= 5.0$ \\ 
$\Delta G_0= 0.0 $ \end{tabular} };
\node [ left= of n1] (n1p) {\begin{tabular}{l} 
RES\\ 
 $D^*_1= 3.0$ \\ 
$ \tilde{a}_1= 15.0$ \\ 
$\Delta G_1= 3.0 $ \end{tabular} };
\node [right= of n2] (n2p) {\begin{tabular}{l} 
RES\\ 
 $D^*_2= 3.0$ \\ 
$ \tilde{a}_2= 10.0$ \\ 
$\Delta G_2= 5.0 $ \end{tabular} };
\draw[ thick, red,text=black,sloped] (n0) -- node [above] {$\kappa_{01}= 10.0$} (n1)  (n0) -- node [above] {$\kappa_{02}= 10.0$} (n2)  (n1) -- node [above] {$\kappa_{12}= 5.0$} (n2)  ;

\end{tikzpicture}

\caption{\small{Three node network example.}}
\label{fig:scheme_3Nodes}
\end{center}
\end{figure}

 Regarding the preferences $(c\nm)\nm$, nodes $1$ and $2$ both prefer to buy local and to RES-based generators. 
 Node $0$ is assumed to be indifferent between buying energy from node $1$ or node $2$. 
Capacities are also defined larger from the source node $0$ ($\kappa_{0n}=10$) than between the prosumers nodes ($\kappa_{nm}=5$). 

\begin{table}[!ht]
\begin{center}
\begin{tabular}{|c|c c c|}\hline
$c\nm$  & $0$ & $1$ & $2$ \\\hline
0 & -- & 1.0 & 1.0\\
1& 3.0 & -- & 1.0\\
2 & 2.0 & 1.0 & -- \\ \hline 
\end{tabular} 
\quad
\begin{tabular}{|c|c c c|}\hline
$c\nm-c\mn $ & $0$ & $1$ & $2$ \\\hline
0 & -- & -2.0 & -1.0\\
1 & 2.0 & -- & 0.0\\
2 & 1.0 & 0.0 & --\\\hline
\end{tabular}

\caption{{\small{Price differentiation parameters and matrix of differences.}}}
\label{tab:parameters}
\end{center}
\end{table}

\newcommand{\hh}{\hspace{-2pt}} 
\tikzstyle{line} = [draw,thick=2, color=green!50, -latex',text=black,sloped ]

\tikzstyle{linec} = [draw,thick=2, color=red!70, -latex',text=black,sloped ]

\begin{figure}[!htbp]
\vspace{-0.5cm}
\begin{center}
\subfloat[Centralized solution ($SW=378.3$)]{
\begin{tikzpicture}[scale=1.0, >=latex]
\node [draw,circle,blue,fill=white,text=black] (n0) at (0,4) {0};

\node [draw,thick,circle, orange,text=black] (n1) at (-2.5,0) {1};

\node [draw,thick,circle, green!80,text=black] (n2) at (2.5,0) {2};
 \node [ right= 0.3cm of n0] (n0p)  {\begin{tabular}{l} 
 $\lambda_0 = 9.34 $  \\
     $ D_0=5.07$ \\
  $ G_0=2.17$ \\
$ Q_0=2.9  $
\end{tabular} };
\node [below= (-0.cm) of n1] (n1p)  {\begin{tabular}{l} 
 $\lambda_1 = 11.34 $  \\
     $ D_1=2.62$ \\
  $ Q_1=-0.38  $
\end{tabular} };
\node [below = (-0.cm) of n2] (n2p)  {\begin{tabular}{l} 
 $\lambda_2 = 10.34 $  \\
     $ D_2=2.48$ \\
  $ Q_2=-2.52  $
\end{tabular} };
\path[line,line width=2.48pt](n0) -- node [] {$q_{02}=2.48$} node [above] { $\zeta_{02}\hh=\hh8.34 $}  node [below,blue ] {$\xi_{20} = 0.0 $} (n2) ;
\path[line,line width=5.38pt](n1) -- node [] {$q_{10}=5.38$} node [above] { $\zeta_{10}\hh=\hh8.34 $}  node [below,blue ] {$\xi_{01} = 0.0 $} (n0) ;
\path[linec,line width=5.0pt](n2) -- node [] {$q_{21}=5.0$} node [above] { $\zeta_{21}\hh=\hh9.34 $}  node [below,blue ] {$\xi_{12} = 1.0 $} (n1) ;
\end{tikzpicture}
}
\subfloat[One GNE ($SW=255.5$)]{ \label{fig:3NodesSolGNE}
\begin{tikzpicture}[scale=1.0, >=latex]
\node [draw,circle,blue,fill=white,text=black] (n0) at (0,4) {0};

\node [draw,thick,circle, orange,text=black] (n1) at (-2.5,0) {1};

\node [draw,thick,circle, green!80,text=black] (n2) at (2.5,0) {2};

 \node [ right= 0.5cm of n0] (n0p)  {\begin{tabular}{l} 
 $\lambda_0 = 1.0 $  \\
     $ D_0=5.9$ \\
  $ G_0=0$ \\
$ Q_0=5.9  $
\end{tabular} };
\node [below= (-0.cm) of n1] (n1p)  {\begin{tabular}{l} 
 $\lambda_1 = 90.0 $  \\
     $ D_1=0.0$ \\
  $ Q_1=-3  $
\end{tabular} };
\node [below = (-0.cm) of n2] (n2p)  {\begin{tabular}{l} 
 $\lambda_2 = 18.0 $  \\
     $ D_2=2.1$ \\
  $ Q_2=-2.9  $
\end{tabular} };
\path[line,line width=2pt](n0) -- node [] {$q_{01}=2$} node [above] { $\zeta_{01}\hh=\hh0,\zeta_{10}\hh=\hh87 $}  node [below,blue ] {$\xi_{10} = 0.0 $} (n1) ;
\path[linec,line width=5.0pt](n1) -- node [] {$q_{12}=5.0$} node [above] { $\zeta_{12}\hh=\hh89,\zeta_{21}\hh=\hh17 $}  node [below,blue ] {$\xi_{21} = 0.0 $} (n2) ;
\path[line,line width=7.9pt](n2) -- node [] {$q_{20}=7.9$} node [above] { $\zeta_{20}\hh=\hh16,\zeta_{02}\hh=\hh0 $}  node [below,blue ] {$\xi_{02} = 0.0 $} (n0) ;

\end{tikzpicture}
}
\caption{\small{Comparison of the optimal centralized solution (a) and a GNE solution with low social welfare (b).}}
\label{fig:3NodesSolCentral}
\end{center}
\end{figure}
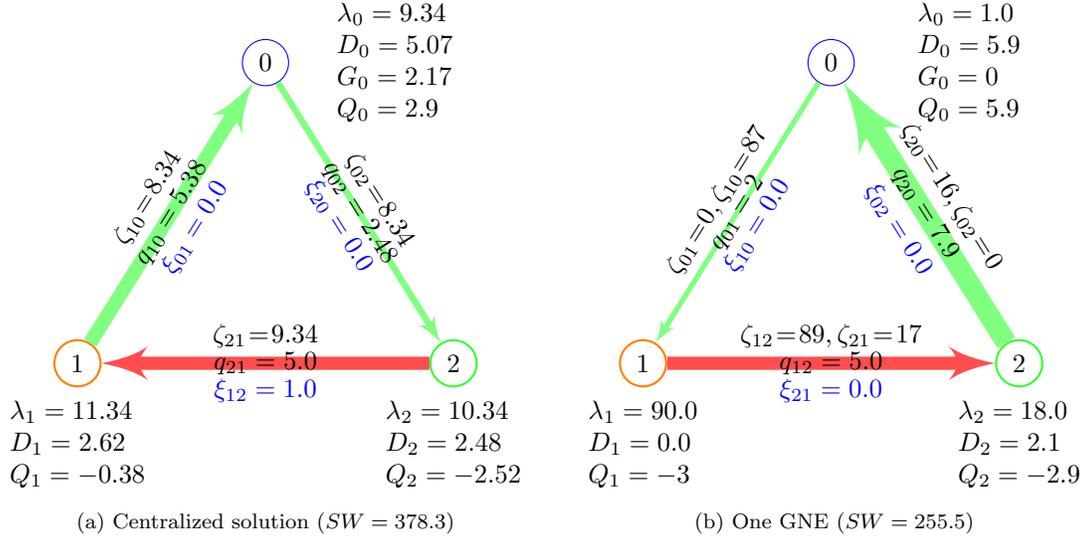

In \Cref{fig:3NodesSolCentral} (a), we illustrate the optimal solution of the centralized market design problem in which the global MO maximizes the social welfare under operational and power-flow constraints \eqref{eq:CM1}-\eqref{eq:CM5}.

We remark on this figure that the trade from node 1 to node 2 is at full capacity, which is explained by  \Cref{prop:fullCapaTrades}.
Indeed,  we see from \Cref{tab:parameters} that there is a ``cycle'' in preferences  $\tC_{01}+\tC_{12}+\tC_{20}=-1$  which explains why we obtain $q_{10}=\kappa_{10}$ and $q_{21}=\kappa_{21}$ in the centralized solution (\Cref{fig:3NodesSolCentral}).

On the contrary, we remark that, in the GNE solution depicted in \Cref{fig:3NodesSolGNE}, the same edge is congested in the reverse way: \Cref{prop:fullCapaTrades} only applies in the case of a centralized solution.

In the example above, the  cycle comes from the fact that it is easier for node 2 to buy from 0 than node 1 to buy from 0: thus, the social welfare can be increased if 1 buys from 2 who buys from 0. Changing the parameters to $c_{20}=c_{30}=3$ removes the cycle in the optimal solution of $(q\nm)\nm$.

\begin{figure}
    \centering
    \includegraphics[width=0.7\textwidth]{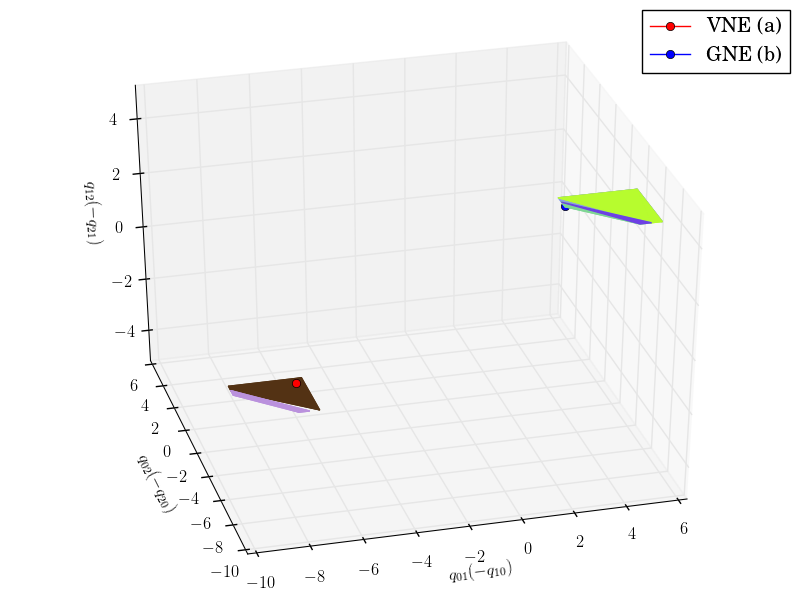}
    \caption{All existing GNEs in $q$-space. \textit{The set of GNEs is given as two connected components, corresponding to the edge (1,2) saturated in one way and the other.} }
    \label{fig:allGNEs3nodes}
\end{figure}
\bigskip

In \Cref{fig:allGNEs3nodes}, we show the different GNEs existing for this reduced problem in the three-dimensional space of transactions. As one can see on this figure, an interesting property is that, for any GNE, the edge from node 1 to node 2 is saturated in one way or the other. 

Evaluating the GNE with the lowest social welfare is difficult because this task does not correspond to a convex problem (in particular, the SW is a concave function). However, the GNE depicted in \Cref{fig:3NodesSolGNE} is the worst GNE that we found with the sampling method given by \Cref{prop:GNEparam}, using a sampling $(\omega\nm)_{n>m}  \in \{ 0,\dots, 100 \}^3$.
Therefore, we can have the following bound on the PoA:
\begin{equation}
    \text{PoA} =  \frac{\max_{\mathbf{D},\mathbf{G},\mathbf{q}} SW(\mathbf{D},\mathbf{G},\mathbf{q})}{\min_{\mathbf{D},\mathbf{G},\mathbf{q} \in \text{SOL}^{\text{GNEP}}} SW(\mathbf{D},\mathbf{G},\mathbf{q})} \geq \frac{378.3}{255.5} \simeq 1.48 \ ,
\end{equation}
which means that, in the peer-to-peer market, in the presence of market imperfections, the resulting social welfare can be more than 50$\%$ smaller than the optimal social welfare (or, the VE obtained in the absence of market imperfections).


\subsection{IEEE 14-bus Network}\label{ex:ex2}

In this example, we consider the IEEE 14-bus network system introduced in \cite{sousa}. Each bus of the network corresponds to a prosumer in our model as described on \Cref{fig:IEEE_network}. The buses 3, 4, 5 and 9 to 14 contain only consumers without any production. Nodes 2 and 3 are prosumers node (consumption and RES production) and also contain thermal production plants. The bus 6 is a prosumer with only intermittent solar energy production. Last, the bus 8 contains only production, renewable and thermal.

The bus 1 corresponding to the grid connection is also  able to provide power to the busses linked to it. 

Each pair of busses is able to trade with its neighboring busses, up to the capacity of the edge linking the pair of busses.

For simplicity, we compute the trades and optimal productions and consumptions for a particular unique time period. 
The renewable energy productions $(\Delta G_n)_n$ and the objective consumptions $(D^{\star}_n)_n $ for this time period are provided in \Cref{fig:IEEE_network}. Note that in this particular example, we have the inequality: 
\begin{equation*}
    28.39 = \sum_{n\in\N} \Delta G_n< \sum_{n\in\N} D^{\star}_n = 69.94 \; \textrm{[GWh]}\ ,
\end{equation*}
which explains partly why we do not have any energy waste in the solutions depicted on \Cref{fig:res_trades_IEEE}.

\begin{figure}
\begin{center}
\begin{tikzpicture}[scale=0.5, >=latex,every node/.style={inner sep=0pt,outer sep=0,font=\small}]
\pic[scale=0.5, >=latex,every node/.style={inner sep=0pt,outer sep=0,font=\small}]{subnodesIEEE};

\tikzstyle{line1v2} = [draw,line width=1pt, color=black,text=black,sloped ]
\tikzstyle{line1v5} = [draw,line width=1pt, color=black,text=black,sloped ]
\tikzstyle{line1v6} = [draw,line width=1pt, color=black,text=black,sloped ]
\tikzstyle{line1v7} = [draw,line width=1pt, color=black,text=black,sloped ]
\tikzstyle{line1v12} = [draw,line width=1pt, color=black,text=black,sloped ]
\tikzstyle{line2v3} = [draw,line width=1pt, color=black,text=black,sloped ]
\tikzstyle{line2v4} = [draw,line width=1pt, color=black,text=black,sloped ]
\tikzstyle{line2v5} = [draw,line width=1pt, color=black,text=black,sloped ]
\tikzstyle{line3v4} = [draw,line width=1pt, color=black,text=black,sloped ]
\tikzstyle{line4v5} = [draw,line width=1pt, color=black,text=black,sloped ]
\tikzstyle{line4v7} = [draw,line width=1pt, color=black,text=black,sloped ]
\tikzstyle{line4v9} = [draw,line width=1pt, color=black,text=black,sloped ]
\tikzstyle{line5v6} = [draw,line width=1pt, color=black,text=black,sloped ]
\tikzstyle{line6v11} = [draw,line width=1pt, color=black,text=black,sloped ]
\tikzstyle{line6v12} = [draw,line width=1pt, color=black,text=black,sloped ]
\tikzstyle{line6v13} = [draw,line width=1pt, color=black,text=black,sloped ]
\tikzstyle{line7v8} = [draw,line width=1pt, color=black,text=black,sloped ]
\tikzstyle{line7v9} = [draw,line width=1pt, color=black,text=black,sloped ]
\tikzstyle{line9v10} = [draw,line width=1pt, color=black,text=black,sloped ]
\tikzstyle{line9v14} = [draw,line width=1pt, color=black,text=black,sloped ]
\tikzstyle{line10v11} = [draw,line width=1pt, color=black,text=black,sloped ]
\tikzstyle{line12v13} = [draw,line width=1pt, color=black,text=black,sloped ]
\tikzstyle{line13v14} = [draw,line width=1pt, color=black,text=black,sloped ]

\node [below=10pt of anchor4v2] (anchor4v2part) {};
\renewcommand{\=}{\hspace{-2pt} = \hspace{-2pt}}

\draw [line1v2] (anchor1v2) -- node [below,yshift=-.1cm] {\tiny $1.0 \quad 1.5$}  (anchor2v1);
\draw [line1v5] (anchor5v1) --++ (0,-30pt) --++ (-30pt,0) -- node [below,yshift=-.1cm] {\tiny $1.0 \quad 1.58$} (anchor1v5) ;
\draw [line1v6] (anchor1v5) -- node [below,yshift=-.1cm] {\tiny $1.0 \quad 1.69$} ++ (6,.8) --++ (0,1.3) ;
\draw [line1v7] (anchor1v7) --  ++ (50pt,-38pt) node [below,yshift=-.1cm,xshift= .5cm] {\tiny $1.0 \quad 1.49$} -| (anchor7v1) ;
\draw [line1v12] (anchor1) -- node [below,yshift=-.1cm] {\tiny $1.0 \quad 1.4$} (anchor12);
\draw [line2v3] (anchor2v3) --++ (0,20pt)-- node [below,yshift=-.1cm] {\tiny $.76 \quad .52$}   ++(3.2,0) --++ (0,-20pt) ;
\draw [line2v4] (anchor2v4) --++ (0,20pt) -- node [below,yshift=-.1cm] {\tiny $.62 \quad .7$}  (anchor4v2part) -- (anchor4v2) ;
\draw [line2v5] (anchor2v6) -- node [below,yshift=-.1cm] {\tiny $.0 \quad .99$}   ++ (0,3.7);
\draw [line3v4] (anchor3v4) -- node [below,xshift=-0.2cm,yshift=-.1cm] {\tiny $.14 \quad .08$}   ++(0,3.7);
\draw [line4v5] (anchor4v5) --++ (0,-20pt) -- node [below,yshift=-.1cm] {\tiny $.78 \quad .02$}  ++(-4,0) --++ (0,20pt);
\draw [line4v7] (anchor4v7) --  node [below,yshift=-.1cm] {\tiny $.72 \quad .33$} ++ (0,3.3) ;
\draw [line4v9] (anchor4v5) -- node [above,yshift=0.1 cm,xshift=0.2cm,] {\tiny $.42 \quad \quad \quad .79$} ++ (0,5.5) ;
\draw [line5v6] (anchor5v2) -- node [below,yshift=-.1cm,xshift=-0.3cm,] {\tiny $.19 \quad \quad \quad .85$}  ++(0,5);
\draw [line6v11] (anchor6v11) -- node [above,yshift=.1cm] {\tiny $.92 \quad .25$} ++ (0,2.5 ) ;
\draw [line6v12] (anchor6v1) -- node [below,yshift=-.1cm] {\tiny $.47 \quad .03$}  (anchor12);
\draw [line6v13] (anchor6v13) -- node [above,yshift=.1cm] {\tiny $.8 \quad \quad \quad .66$} ++ (0,3.9) ;
\draw [line7v8] (anchor7v8) --++(0,35pt) -- node [below,yshift=-.1cm] {\tiny $.78 \quad .13$}   ++(1.4,0);
\draw [line7v9] (anchor7v1) -- node [below,yshift=-.1cm] {\tiny $.24 \quad .38$} ++ (0,2.2) ;
\draw [line9v10] (anchor9v4) -- node [above,yshift=.1cm] {\tiny $.48 \quad .97$} ++ (0,2.1) ;
\draw [line9v14] (anchor9v14) -- node [below,xshift=0.2cm,yshift=-.1cm] {\tiny $.91 \quad .02$} ++ (0,3.5) ;
\draw [line10v11] (anchor10v11) --++(0,20pt) -- node [below,yshift=-.1cm] {\tiny $.46 \quad .32$}  ++(-3.5,0)  --++(0,-20pt) ;
\draw [line12v13] (anchor12v13) -- node [above,yshift=.1cm] {\tiny $.35 \quad .76$}  (anchor13v12) ;
\draw [line13v14] (anchor13v14) --++(0,20pt) -- node [above,yshift=.1cm] {\tiny $.51 \quad .52$}   ++ (3.8,0) --++(0,-20pt);

\def\nd{0.3}
\def \idn{0.15} 
\node [above=0.5cm of G0] (tg1) {\tiny Grid connection}; 
\node [ below=0.8cm of g2,rotate=90] (tg4) {\tiny \begin{tabular}{c}
     wind, solar \\
     $\Delta G\=0.40$
\end{tabular}  }; 
\node [ below=0.5cm of g4,rotate=90] (tg4) {\tiny gas}; 
\node [ below=0.8cm of g6,rotate=90] (tg4) {\tiny \begin{tabular}{c}
     wind \\
     $\Delta G\=4.99$
\end{tabular}  }; 
\node [ below=0.5cm of g7,rotate=90] (tg7) {\tiny coal}; 
\node [ right=0.2cm of g10] (tg10) {\tiny \begin{tabular}{c}
     wind \\
     $\Delta G\=7.5$
\end{tabular}  }; 
\node [ right=0.5cm of g11] (tg11) {\tiny gas}; 
\node [ below=0.5cm of g15,rotate=90] (tg15) {\tiny \begin{tabular}{c}
      solar \\
     $\Delta G\=15.51$
\end{tabular}  }; 

\node [ below=0.8cm of c2,rotate=90] (tc2) {\tiny $D^*_{2} \=6.57$} ;
\node [ below=0.8cm of c3,rotate=90] (tc3) {\tiny $D^*_{3} \=12.55$} ;
\node [right= 0.2cm of c4] (tc4) {\tiny $D^*_{4} \=8.75$} ;
\node [left= 0.2cm of c5] (tc5) {\tiny $D^*_{5} \=6.37$} ;
\node [ below=0.55cm of c6,xshift=-0.1cm,rotate=90] (tc6) {\tiny $D^*_{6} \=4.33$} ;
\node [right= 0.4cm of c9] (tc9) {\tiny $D^*_{9} \=9.42$} ;
\node [ below=0.8cm of c10,rotate=90] (tc10) {\tiny $D^*_{10} \=3.27$} ;
\node [below right= 0.1cm of c11] (tc11) {\tiny $D^*_{11} \=4.51$} ;
\node [above= 0.8cm of c12,rotate=90] (tc12) {\tiny $D^*_{12} \=3.26$} ;
\node [above= 0.8cm of c13,rotate=90] (tc13) {\tiny $D^*_{13} \=5.63$} ;
\node [above= 0.8cm of c14,rotate=90] (tc14) {\tiny $D^*_{14} \=5.28$} ;


\node [below= 5cm of G0, xshift=-3cm ] (ln) {$n$} ;
\node [right= 2cm of ln] (lm) {$m$} ;
\draw (ln) -- node [above, yshift=.1cm] {\tiny $c_{nm} \quad c_{mn}$} (lm) ;

\end{tikzpicture}

\caption{\small{IEEE 14-bus network system }}
\label{fig:IEEE_network}
\end{center}
\end{figure}
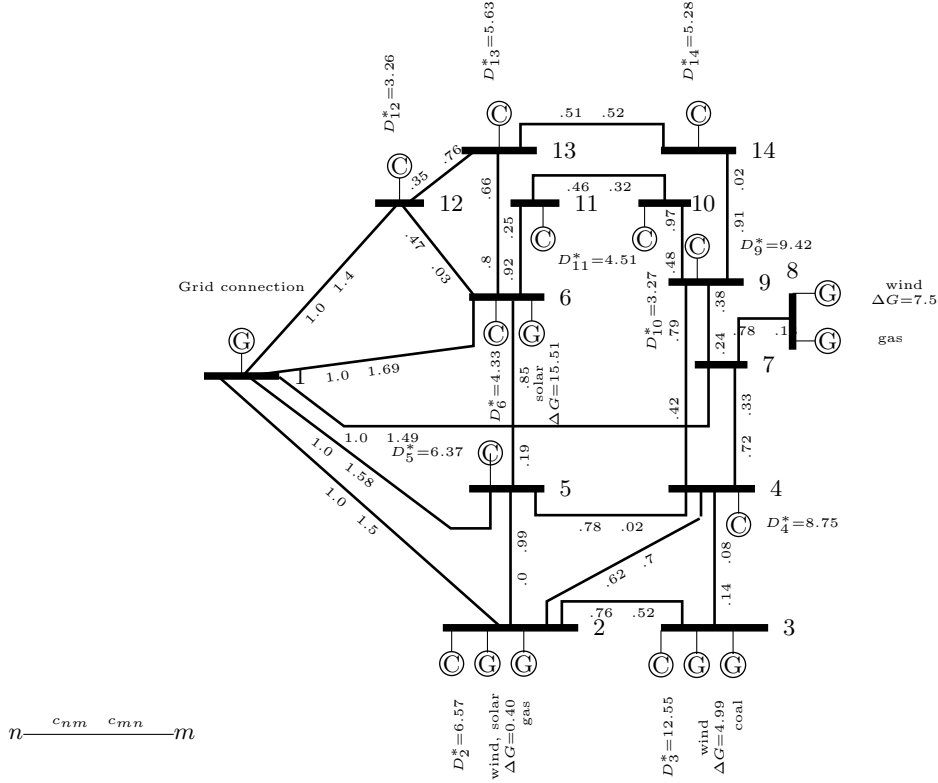


For the trade differentiation prices $(c_{nm})_{n,m}$, we consider four cases:
\begin{enumerate}[label=(\alph*)]
    \item uniform prices: $c_{nm}=1$ for each $n$ and $m$, so that we ensure that there does not exist any cycle in the matrix of price differences as described in \Cref{prop:fullCapaTrades};
    \item heterogeneous prices: for $n\neq 1$ and $m \neq 1$, $c_{nm}$ is chosen uniformly in $[0,1]$. We assume that agents have a preference for local trades so the price with the grid connection bus $c_{n1}$ is larger and chosen uniformly in $[1,2]$. The grid connection bus has no preferences so that $c_{1n}=1$ for each $n$ neighboring bus 1. 
    \item symmetric prices: $(c_{nm})_{nm}$ random and symmetric (for $n< m$, $c_{nm}$ is taken as in (b)).
    \item preferences for local trades with uniform prices: $(c_{nm})_{nm}=1$ if $m\neq 1$ and $c_{n1}=3$.
\end{enumerate}
For each of this case, we compute the centralized solution (also corresponding to the VNE). The solutions are illustrated in \Cref{fig:IEEE_network}: directions of trades are represented by arrows, the wideness of each arrow is proportional to the quantity traded. Trades made at full capacity ($q\nm=\kappa\nm$) are represented by red arrows, while the others are represented by green arrows. We observe that cases (c) and (d) give the same trade solutions $(q\nm)\nm $ at VNE as case (a).

\vspace{0cm}
\begin{figure}
\subfloat[with $(c\nm)\nm =(1)\nm$ uniform, $\SW=395.28$ M\$]{
\begin{tikzpicture}[scale=0.45, >=latex,every node/.style={inner sep=0pt,outer sep=0,font=\small}]

\pic[scale=0.45, >=latex,every node/.style={inner sep=0pt,outer sep=0,font=\small}]{subnodesIEEE};
\node [below=10pt of anchor4v2] (anchor4v2part) {};
\renewcommand{\=}{\hspace{-2pt} = \hspace{-2pt}}

\tikzstyle{line1v2} = [draw,line width=0.73pt, color=green!70, -latex',text=black,sloped ]
\tikzstyle{line1v5} = [draw,line width=0.34pt, color=green!70, -latex',text=black,sloped ]
\tikzstyle{line1v6} = [draw,line width=0.97pt, color=green!70, latex'-,text=black,sloped ]
\tikzstyle{line1v7} = [draw,line width=0.23pt, color=green!70, latex'-,text=black,sloped ]
\tikzstyle{line1v12} = [draw,line width=0.13pt, color=green!70, -latex',text=black,sloped ]
\tikzstyle{line2v3} = [draw,line width=0.11pt, color=green!70, -latex',text=black,sloped ]
\tikzstyle{line2v4} = [draw,line width=0.18pt, color=green!70, -latex',text=black,sloped ]
\tikzstyle{line2v5} = [draw,line width=0.04pt, color=green!70, latex'-,text=black,sloped ]
\tikzstyle{line3v4} = [draw,line width=0.01pt, color=green!70, -latex',text=black,sloped ]
\tikzstyle{line4v5} = [draw,line width=0.13pt, color=green!70, latex'-,text=black,sloped ]
\tikzstyle{line4v7} = [draw,line width=0.66pt, color=green!70, latex'-,text=black,sloped ]
\tikzstyle{line4v9} = [draw,line width=0.39pt, color=green!70, -latex',text=black,sloped ]
\tikzstyle{line5v6} = [draw,line width=0.35pt, color=green!70, latex'-,text=black,sloped ]
\tikzstyle{line6v11} = [draw,line width=0.35pt, color=green!70, -latex',text=black,sloped ]
\tikzstyle{line6v12} = [draw,line width=0.22pt, color=green!70, -latex',text=black,sloped ]
\tikzstyle{line6v13} = [draw,line width=0.53pt, color=green!70, -latex',text=black,sloped ]
\tikzstyle{line7v8} = [draw,line width=1.55pt, color=green!70, latex'-,text=black,sloped ]
\tikzstyle{line7v9} = [draw,line width=0.66pt, color=green!70, -latex',text=black,sloped ]
\tikzstyle{line9v10} = [draw,line width=0.29pt, color=green!70, -latex',text=black,sloped ]
\tikzstyle{line9v14} = [draw,line width=0.33pt, color=green!70, -latex',text=black,sloped ]
\tikzstyle{line10v11} = [draw,line width=0.0pt, color=green!70, latex'-,text=black,sloped ]
\tikzstyle{line12v13} = [draw,line width=0.05pt, color=green!70, -latex',text=black,sloped ]
\tikzstyle{line13v14} = [draw,line width=0.1pt, color=green!70, -latex',text=black,sloped ]
\draw [line1v2] (anchor1v2) -- node [below] {\tiny $\zeta\=2.16$}  (anchor2v1);
\draw [line1v5] (anchor5v1) --++ (0,-30pt) --++ (-30pt,0) -- node [below] {\tiny $\zeta\=2.16$} (anchor1v5) ;
\draw [line1v6] (anchor1v5) -- node [below] {\tiny $\zeta\=2.16$} ++ (6,0.8) --++ (0,1.3) ;
\draw [line1v7] (anchor1v7) -- node [below] {\tiny $\zeta\=2.16$} ++ (50pt,-38pt) -| (anchor7v1) ;
\draw [line1v12] (anchor1) -- node [below] {\tiny $\zeta\=2.16$} (anchor12);
\draw [line2v3] (anchor2v3) --++ (0,20pt)-- node [below] {\tiny $\zeta\=2.16$}   ++(3.2,0) --++ (0,-20pt) ;
\draw [line2v4] (anchor2v4) --++ (0,20pt) -- node [below] {\tiny $\zeta\=2.16$}  (anchor4v2part) -- (anchor4v2) ;
\draw [line2v5] (anchor2v6) -- node [below] {\tiny $\zeta\=2.16$}   ++ (0,3.7);
\draw [line3v4] (anchor3v4) -- node [below] {\tiny $\zeta\=2.16$}   ++(0,3.7);
\draw [line4v5] (anchor4v5) --++ (0,-20pt) -- node [below] {\tiny $\zeta\=2.16$}  ++(-4,0) --++ (0,20pt);
\draw [line4v7] (anchor4v7) --  node [below] {\tiny $\zeta\=2.16$} ++ (0,3.3) ;
\draw [line4v9] (anchor4v5) -- node [below] {\tiny $\zeta\=2.16$} ++ (0,5.5) ;
\draw [line5v6] (anchor5v2) -- node [below] {\tiny $\zeta\=2.16$}  ++(0,5);
\draw [line6v11] (anchor6v11) -- node [below] {\tiny $\zeta\=2.16$} ++ (0,2.5 ) ;
\draw [line6v12] (anchor6v1) -- node [below] {\tiny $\zeta\=2.16$}  (anchor12);
\draw [line6v13] (anchor6v13) -- node [below] {\tiny $\zeta\=2.16$} ++ (0,3.9) ;
\draw [line7v8] (anchor7v8) --++(0,35pt) -- node [below] {\tiny $\zeta\=2.16$}   ++(1.4,0);
\draw [line7v9] (anchor7v1) -- node [below] {\tiny $\zeta\=2.16$} ++ (0,2.2) ;
\draw [line9v10] (anchor9v4) -- node [above] {\tiny $\zeta\=2.16$} ++ (0,2.1) ;
\draw [line9v14] (anchor9v14) -- node [below] {\tiny $\zeta\=2.16$} ++ (0,3.5) ;
\draw [line10v11] (anchor10v11) --++(0,20pt) -- node [below] {\tiny $\zeta\=2.16$}  ++(-3.5,0)  --++(0,-20pt) ;
\draw [line12v13] (anchor12v13) -- node [below] {\tiny $\zeta\=2.16$}  (anchor13v12) ;
\draw [line13v14] (anchor13v14) --++(0,20pt) -- node [below] {\tiny $\zeta\=2.16$}   ++ (3.8,0) --++(0,-20pt);

\node [above= 0.2cm of G0] (tc1) {\tiny \textcolor{blue}{ $\lambda_{1} \= 3.16$} } ;
\node [ below=0.2cm of c2] (tc2) {\tiny \textcolor{blue}{ $\lambda_{2} \= 3.16$} } ;
\node [ below=0.2cm of c3] (tc3) {\tiny \textcolor{blue}{ $\lambda_{3} \= 3.16$} } ;
\node [right= 0.2cm of c4] (tc4) {\tiny \textcolor{blue}{ $\lambda_{4} \= 3.16$} } ;
\node [left= 0.2cm of c5] (tc5) {\tiny \textcolor{blue}{ $\lambda_{5} \= 3.16$} } ;
\node [left= 0.2cm of c6, yshift=0.3cm] (tc6) {\tiny \textcolor{blue}{ $\lambda_{6} \= 3.16$} } ;
\node [right= 0.2cm of g11, rotate=90] (tc8) {\tiny \textcolor{blue}{ $\lambda_{8} \= 3.16$} } ;
\node [right= 0.6cm of c9] (tc9) {\tiny \textcolor{blue}{ $\lambda_{9} \= 3.16$} } ;
\node [right= 1.2cm of c10, yshift=0.5cm] (tc10) {\tiny \textcolor{blue}{ $\lambda_{10} \= 3.16$} } ;
\node [below=0.1cm of c11,xshift=0.6cm] (tc11) {\tiny \textcolor{blue}{ $\lambda_{11} \= 3.16$} } ;
\node [above= 0.2cm of c12] (tc12) {\tiny \textcolor{blue}{ $\lambda_{12} \= 3.16$} } ;
\node [above= 0.2cm of c13] (tc13) {\tiny \textcolor{blue}{ $\lambda_{13} \= 3.16$} } ;
\node [above= 0.2cm of c14] (tc14) {\tiny \textcolor{blue}{ $\lambda_{14} \= 3.16$} } ;
\end{tikzpicture}
}
\subfloat[with $(c_{nm})_{nm}$ random, $\SW=560.51$ M\$]{
\begin{tikzpicture}[scale=0.45, >=latex,every node/.style={inner sep=0pt,outer sep=0,font=\small}]

\pic[scale=0.45, >=latex,every node/.style={inner sep=0pt,outer sep=0,font=\small}]{subnodesIEEE};

\node [below=10pt of anchor4v2] (anchor4v2part) {};
\renewcommand{\=}{\hspace{-2pt} = \hspace{-2pt}}

\tikzstyle{line1v2} = [draw,line width=4.12pt, color=green!70, -latex',text=black,sloped ]
\tikzstyle{line1v5} = [draw,line width=3.49pt, color=green!70, latex'-,text=black,sloped ]
\tikzstyle{line1v6} = [draw,line width=6.42pt, color=green!70, -latex',text=black,sloped ]
\tikzstyle{line1v7} = [draw,line width=3.85pt, color=green!70, latex'-,text=black,sloped ]
\tikzstyle{line1v12} = [draw,line width=3.2pt, color=red!70, latex'-,text=black,sloped ]
\tikzstyle{line2v3} = [draw,line width=3.6pt, color=red!70, -latex',text=black,sloped ]
\tikzstyle{line2v4} = [draw,line width=5.03pt, color=green!70, -latex',text=black,sloped ]
\tikzstyle{line2v5} = [draw,line width=5.0pt, color=red!70, latex'-,text=black,sloped ]
\tikzstyle{line3v4} = [draw,line width=3.49pt, color=green!70, -latex',text=black,sloped ]
\tikzstyle{line4v5} = [draw,line width=4.5pt, color=red!70, -latex',text=black,sloped ]
\tikzstyle{line4v7} = [draw,line width=5.5pt, color=red!70, -latex',text=black,sloped ]
\tikzstyle{line4v9} = [draw,line width=2.08pt, color=green!70, latex'-,text=black,sloped ]
\tikzstyle{line5v6} = [draw,line width=4.5pt, color=red!70, latex'-,text=black,sloped ]
\tikzstyle{line6v11} = [draw,line width=1.55pt, color=green!70, -latex',text=black,sloped ]
\tikzstyle{line6v12} = [draw,line width=2.3pt, color=green!70, -latex',text=black,sloped ]
\tikzstyle{line6v13} = [draw,line width=0.49pt, color=green!70, -latex',text=black,sloped ]
\tikzstyle{line7v8} = [draw,line width=1.55pt, color=green!70, latex'-,text=black,sloped ]
\tikzstyle{line7v9} = [draw,line width=3.2pt, color=red!70, -latex',text=black,sloped ]
\tikzstyle{line9v10} = [draw,line width=0.92pt, color=green!70, latex'-,text=black,sloped ]
\tikzstyle{line9v14} = [draw,line width=1.65pt, color=green!70, -latex',text=black,sloped ]
\tikzstyle{line10v11} = [draw,line width=1.2pt, color=red!70, latex'-,text=black,sloped ]
\tikzstyle{line12v13} = [draw,line width=1.2pt, color=red!70, latex'-,text=black,sloped ]
\tikzstyle{line13v14} = [draw,line width=1.2pt, color=red!70, latex'-,text=black,sloped ]
\draw [line1v2] (anchor1v2) -- node [below] {\tiny $\zeta\=1.49$}  (anchor2v1);
\draw [line1v5] (anchor5v1) --++ (0,-30pt) --++ (-30pt,0) -- node [below] {\tiny $\zeta\=1.49$} (anchor1v5) ;
\draw [line1v6] (anchor1v5) -- node [below] {\tiny $\zeta\=1.49$} ++ (6,0.8) --++ (0,1.3) ;
\draw [line1v7] (anchor1v7) -- node [below] {\tiny $\zeta\=1.49$} ++ (50pt,-38pt) -| (anchor7v1) ;
\draw [line1v12] (anchor1) -- node [below] {\tiny $\zeta\=1.34$} (anchor12);
\draw [line2v3] (anchor2v3) --++ (0,20pt)-- node [below] {\tiny $\zeta\=2.23$}   ++(3.2,0) --++ (0,-20pt) ;
\draw [line2v4] (anchor2v4) --++ (0,20pt) -- node [below] {\tiny $\zeta\=2.37$}  (anchor4v2part) -- (anchor4v2) ;
\draw [line2v5] (anchor2v6) -- node [below] {\tiny $\zeta\=2.08$}   ++ (0,3.7);
\draw [line3v4] (anchor3v4) -- node [below] {\tiny $\zeta\=2.99$}   ++(0,3.7);
\draw [line4v5] (anchor4v5) --++ (0,-20pt) -- node [below] {\tiny $\zeta\=2.29$}  ++(-4,0) --++ (0,20pt);
\draw [line4v7] (anchor4v7) --  node [below] {\tiny $\zeta\=2.35$} ++ (0,3.3) ;
\draw [line4v9] (anchor4v5) -- node [below] {\tiny $\zeta\=2.65$} ++ (0,5.5) ;
\draw [line5v6] (anchor5v2) -- node [below] {\tiny $\zeta\=2.33$}  ++(0,5);
\draw [line6v11] (anchor6v11) -- node [below] {\tiny $\zeta\=2.26$} ++ (0,2.5 ) ;
\draw [line6v12] (anchor6v1) -- node [below] {\tiny $\zeta\=2.71$}  (anchor12);
\draw [line6v13] (anchor6v13) -- node [below] {\tiny $\zeta\=2.38$} ++ (0,3.9) ;
\draw [line7v8] (anchor7v8) --++(0,35pt) -- node [below] {\tiny $\zeta\=2.2$}   ++(1.4,0);
\draw [line7v9] (anchor7v1) -- node [below] {\tiny $\zeta\=2.74$} ++ (0,2.2) ;
\draw [line9v10] (anchor9v4) -- node [above] {\tiny $\zeta\=2.96$} ++ (0,2.1) ;
\draw [line9v14] (anchor9v14) -- node [below] {\tiny $\zeta\=2.53$} ++ (0,3.5) ;
\draw [line10v11] (anchor10v11) --++(0,20pt) -- node [below] {\tiny $\zeta\=2.19$}  ++(-3.5,0)  --++(0,-20pt) ;
\draw [line12v13] (anchor12v13) -- node [below] {\tiny $\zeta\=2.28$}  (anchor13v12) ;
\draw [line13v14] (anchor13v14) --++(0,20pt) -- node [below] {\tiny $\zeta\=2.03$}   ++ (3.8,0) --++(0,-20pt);

\node [above= 0.2cm of G0] (tc1) {\tiny \textcolor{blue}{ $\lambda_{1} \= 2.55$} } ;
\node [ below=0.2cm of c2] (tc2) {\tiny \textcolor{blue}{ $\lambda_{2} \= 2.49$} } ;
\node [ below=0.2cm of c3] (tc3) {\tiny \textcolor{blue}{ $\lambda_{3} \= 2.99$} } ;
\node [right= 0.2cm of c4] (tc4) {\tiny \textcolor{blue}{ $\lambda_{4} \= 3.13$} } ;
\node [left= 0.2cm of c5] (tc5) {\tiny \textcolor{blue}{ $\lambda_{5} \= 3.07$} } ;
\node [left= 0.2cm of c6, yshift=0.3cm] (tc6) {\tiny \textcolor{blue}{ $\lambda_{6} \= 3.07$} } ;
\node [right= 0.2cm of g11, rotate=90] (tc8) {\tiny \textcolor{blue}{ $\lambda_{8} \= 2.98$} } ;
\node [right= 0.6cm of c9] (tc9) {\tiny \textcolor{blue}{ $\lambda_{9} \= 2.33$} } ;
\node [right= 1.2cm of c10, yshift=0.5cm] (tc10) {\tiny \textcolor{blue}{ $\lambda_{10} \= 3.44$} } ;
\node [below=0.1cm of c11,xshift=0.6cm] (tc11) {\tiny \textcolor{blue}{ $\lambda_{11} \= 3.93$} } ;
\node [above= 0.2cm of c12] (tc12) {\tiny \textcolor{blue}{ $\lambda_{12} \= 2.51$} } ;
\node [above= 0.2cm of c13] (tc13) {\tiny \textcolor{blue}{ $\lambda_{13} \= 2.74$} } ;
\node [above= 0.2cm of c14] (tc14) {\tiny \textcolor{blue}{ $\lambda_{14} \= 3.04$} } ;
\end{tikzpicture}

}
\caption{Trades [\$/MWh] at the VNE of the IEEE 14-bus network with homogeneous differentiation prices (left) and heterogeneous differentiation prices (right). \textit{With heterogeneous prices, the quantities traded are larger, and some links become congested. In the heterogeneous case, marginal trade prices $(\zeta\nm)_{n,m}$ are all equal. In the case of heterogeneous prices $(c\nm)\nm$, marginal prices $(\zeta\nm)_{n,m}$   are also heterogeneous.}}
\label{fig:res_trades_IEEE}
\end{figure}
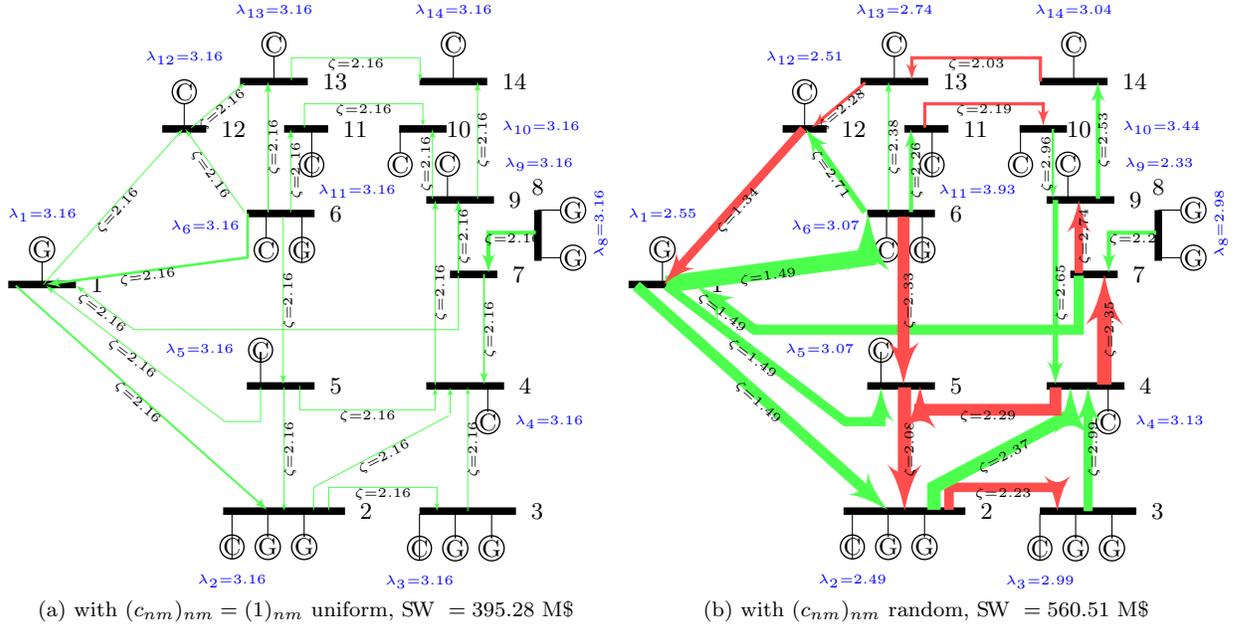
We see in \Cref{fig:res_trades_IEEE} that the differentiation prices $(c\nm)\nm$ modify completely the solution. We observe that the quantities traded in case $(b)$ are much larger. While some edges are almost unused in case (a) and no edge is congested, ten of the twenty-two edges become congested in case (b) with heterogeneous prices. This effect can be explained by Propositions \ref{prop:fullCapaTrades} and \ref{prop:preference_asymmetry}.

Also, we observe that  marginal prices $(\zeta\nm)_{n,m}$ are all equal to 2.16 \$/MWh in case (a), while they are heterogeneous in the case (b). In case (a), the equality is explained by both the absence of congestion ($\xi\nm=0$) and the equality of $(c\nm)\nm$ among users (absence of preferences).

As opposed to the reduced example with three nodes given in \Cref{subsec:3nodeEx}, it was not possible to compute a GNE different from the VE for this 14 nodes network. The approach of Nabetani et al. \cite{nabetani2009} that we used for the three node network is not possible here because of the  dimension: to search for another GNE, we have to look on a space of dimension 22, e.g., the number of lines in the network. This observation also calls for the development of algorithms not based on brute force approach, enabling an efficient approximation of the GNEs. This could be the topic of future research.

\section{Dealing with Privacy}\label{sec:privacy}

Privacy might limit the information released by the prosumers regarding their target demand and RES-based self-generation. In this section, the target demand and RES-based self-generation will be called prosumer's private information. When privacy applies, prosumers can decide to not share their private information with the other prosumers in their neighborhood. Two reasons justify this behavior: first, they might be reluctant to install intrusive and costly monitoring systems to keep track of their RES-based self-generation; second, they can  fear that other prosumers decide to sell it to aggregators that would use it for gaming on the wholesale market, impacting the prosumers' bills. In general, privacy has some impact on the utility functions of the prosumers. Indeed, based on the analytical expressions of the VNE and GNE we derived in the previous sections, each prosumer needs to know the target demands and RES-based self-generation of each prosumer in her neighborhood. In case privacy applies, these values will not be shared within the neighborhood and each prosumer would have to forecast the private information of the other prosumers to optimize her own decision variables, therefore introducing bias in her own expected utility function. In the present section, we provide an upper bound on the impact of the bias introduced by the distributed forecasts of the prosumers on their expected utility. This will give us a first quantification of the impact of privacy on the market outcome taking the point of view of the prosumers. This also constitutes a first step towards the implementation of learning methods or algorithms to approach an equilibrium (GNE)  with minimum information exchange between the prosumers. These algorithmic aspects leave an interesting direction for future research.     

\subsection{Quantifying the Loss Caused by Privacy}

We now explain in more details how each prosumer builds forecasts of target demands, RES-based generations and nodal prices, and use them to compute a biased-forecast equilibrium.

The expressions of the prosumers' demand, flexibility activation, and net imports are given in Proposition~\ref{prop:optimum}. From these expressions, each prosumer needs to compute her nodal price, which is itself based on the nodal price at the root node $\lambda_0$. 
In a centralized market clearing approach, it is the MO who determines all the nodal prices while having access to all the information of the prosumers on their target demand and RES generations. In the reality, privacy preservation rules might allow the prosumers not to share all their private information. In a peer-to-peer market design, nodal price expressions are detailed in \eqref{eq:lambdan_GNE} and \eqref{eq:lambda0_GNE}. Under \textbf{Assumptions~1, 2, 3}, to compute her nodal price, each prosumer needs to compute the nodal price at the root node $\lambda_0$, which requires to know the target demand $D_m^{\star}, m \neq n$ and RES-based generations $\Delta G_m, m \neq n$ of all the other prosumers in her neighborhood. Since this information is in general kept private by the prosumers, prosumer $n$ needs to build forecasts of her neighbors' target demand and RES-based generations. To that purpose, for each agent $n \in \mathcal{N}$, we introduce forecasts in the form of simple linear estimates:
\begin{align}
F_n(D_m^{\star}) =& D_m^{\star} + \epsilon_{nm}^D, \label{eq:biasD} \\
F_n\Big(\Delta G_m\Big) =& \Delta G_m + \epsilon_{nm}^G, \forall m \in \Omega_n, m \neq n, \label{eq:biasG}
\end{align}
where $\epsilon_{nm}^D$, $\epsilon_{nm}^G$ are the biases introduced by agent $n$ in the estimation of the demand and the RES-based generation of any agent $m$ in her neighborhood. We assume that $\epsilon_{nm}^D$, $\epsilon_{nm}^G$ are independent and identically distributed (iid) random variables that follows Gaussian density functions centered in $0$, with standard deviation $\sigma_{nm}^D$, $\sigma_{nm}^{G}$. We also set $\Delta \epsilon_{nm} := \epsilon_{nm}^D-\epsilon_{nm}^G$, as the difference between the biases introduced by agent $n$ in agent $m$ demand and RES-based generation estimations. To simplify the analytical expressions to come, let $\rho_n(\mathbf{r}) := \frac{r_n}{(\frac{1}{2\tilde{a}_0}+\frac{1}{a_0})+\sum_{m\in\Omega_0}(\frac{1}{2\tilde{a}_m}+\frac{1}{a_m})r_m}$.

Substituting Equations~\eqref{eq:biasD} and \eqref{eq:biasG} in \eqref{eq:lambda0_GNE} and \eqref{eq:lambdan_GNE}, agent $n$ obtains the following estimate for the nodal price at the root node and at her node (i.e., on her local market):
\begin{align*}
F_n(\lambda_0) =& \lambda_0 + \rho_n(\mathbf{r}) r_n^{-1} \sum_{m \in \Omega_n}\Delta \epsilon_{nm}, \\
 F_n(\lambda_n) =& \lambda_n + \rho_n(\mathbf{r}) \sum_{m \in \Omega_n} \Delta \epsilon_{nm}, \forall n \in \Omega_0.
\end{align*}

For any prosumer $n \in \mathcal{N}$, we observe that $\sum_{m \in \Omega_n}\Delta \epsilon_{nm} \rightarrow 0$ implies that $F_n(\lambda_0) \rightarrow \lambda_0$ and $F_n(\lambda_n) \rightarrow \lambda_n$. So, agent $n$ estimates of her nodal price is without bias if she makes no bias in the other agents' target demand and RES-based generation estimations, or biases in both estimates compensate each other. 

Then, by substitution of the nodal price estimates in Proposition~\ref{prop:optimum} output, we obtain the following expression for the biased-forecast equilibrium:
\begin{align}
F_n(D_n) =& D_n - \frac{1}{2\tilde{a}_n}\rho_n(\mathbf{r})\sum_{m\in\Omega_n}\Delta\epsilon_{nm}, \nonumber \\
F_n(G_n) =& G_n + \frac{1}{a_n} \rho_n(\mathbf{r}) \sum_{m \in \Omega_n} \Delta \epsilon_{nm}, \nonumber \\
F_n(Q_n) =& Q_n -(\frac{1}{2\tilde{a}_n}+\frac{1}{a_n})\rho_n(\mathbf{r}) \sum_{m\in \Omega_n}\Delta \epsilon_{nm}. \label{eq:estimatedGNE}
\end{align}

Note that in general $F_n(D_n) \neq D_n$, $F_n(G_n) \neq G_n$, and $F_n(Q_n) \neq Q_n$, i.e., the equilibrium obtained under privacy (that we called biased-forecast equilibrium) is different from the equilibrium computed under full information. In case where the sum of the error differences tends to zero, i.e., $\sum_{m\in\Omega_n}\Delta\epsilon_{nm} \rightarrow 0$, then $F_n(D_n) \rightarrow D_n$, $F_n(G_n) \rightarrow G_n$, and $F_n(Q_n) \rightarrow Q_n$.  

Then, by substitution of the biased-forecast equilibrium \eqref{eq:estimatedGNE} in agent $n$ utility function, we obtain the following bounds for her utility bias:
\begin{align}
F_n(\Pi_n) - \Pi_n \leq -\frac{1}{2}(\frac{1}{\tilde{a}_n}-\frac{1}{a_n})\rho_n(\mathbf{r})^2(\sum_{m \in \Omega_n}\Delta \epsilon_{nm})^2+\Big( \min_{m \neq n}\{c_{nm}\} (\frac{1}{2\tilde{a}_n}+\frac{1}{a_n})-\frac{b_n}{a_n}\Big)\rho_n(\mathbf{r}) \sum_{m\in \Omega_n}\Delta \epsilon_{nm}, \nonumber \\
F_n(\Pi_n) - \Pi_n \geq -\frac{1}{2}(\frac{1}{\tilde{a}_n}-\frac{1}{a_n})\rho_n(\mathbf{r})^2(\sum_{m \in \Omega_n}\Delta \epsilon_{nm})^2+\Big( \max_{m \in \Omega_n}\{c_{nm}\} (\frac{1}{2\tilde{a}_n}+\frac{1}{a_n})-\frac{b_n}{a_n}\Big)\rho_n(\mathbf{r}) \sum_{m\in \Omega_n}\Delta \epsilon_{nm}.\label{eq:boundPi}
\end{align}

Taking the expectation of $F_n(\Pi_n)-\Pi_n$ and since the expectation preserves the inequalities, we obtain the following relation from Equation~\eqref{eq:boundPi}:
\begin{equation}\label{eq:bias_gap}
\mathbb{E}\Big[F_n(\Pi_n)-\Pi_n\Big] = -\frac{1}{2}(\frac{1}{\tilde{a}_n}-\frac{1}{a_n})\rho_n(\mathbf{r})^2\sum_{m \in \Omega_n}\Big( (\sigma_{nm}^D)^2+(\sigma_{nm}^G)^2+2 Cov(\epsilon_{nm}^D,\epsilon_{nm}^G)\Big).
\end{equation}

If $\tilde{a}_n = a_n$, then $\mathbb{E}\Big[F_n(\Pi_n)-\Pi_n\Big]=0$, i.e., there is no bias in the estimation of the prosumer's utility.

To simplify the notation, we set $\beta_n := -(\frac{1}{\tilde{a}_n}-\frac{1}{a_n})\sum_{m\in \Omega_n}\Big( (\sigma_{nm}^D)^2+(\sigma_{nm}^G)^2+2Cov(\epsilon_{nm}^D,\epsilon_{nm}^G)\Big)$. 

\begin{proposition}\label{prop:biasUtility}
Assuming $\beta_n>0$ (resp. $\beta_n<0$), the bias in the prosumer $n$ estimated utility is increasing in $r_n$ (resp. decreasing in $r_n$). 
\end{proposition}
\begin{proof}
By derivation of node $n$ profit with respect to $\delta_{n0}$, we obtain:
$\frac{\partial \mathbb{E}\Big[F_n(\Pi_n)-\Pi_n\Big]}{\partial \delta_{n0}} = 
\beta_n\rho_n(\mathbf{r})\frac{\partial \rho_n(\mathbf{r})}{\partial r_n}$. 

Since $\frac{\partial \rho_n(\mathbf{r})}{\partial r_n} = \frac{\rho_n(\mathbf{r})}{r_n} \Big[ \frac{1}{\delta_{n0}}-\rho_n(\mathbf{r})(\frac{1}{2\tilde{a}_n}+\frac{1}{a_n})\Big]=\frac{\rho_n(\mathbf{r})}{r_n}   (1- \frac{r_n \alpha_n}{ \alpha_0 + \sum_{m\in\Omega_0} \alpha_m r_m }) $ with $\alpha_m := \frac{1}{2 \tilde{a}_m} + \frac{1}{ a_m}, \forall m \in \Omega_0$. By assumption $\zeta_{n0}\geq 0$ and $\zeta_{0n}\geq 0$, which imply that $r_n \geq 0$ and $\rho_n(\mathbf{r})\geq 0$. Furthermore, by definition of the usage benefit and production cost parameters $\alpha_0+\sum_{m\in\Omega_0\setminus n}\alpha_m r_m \geq 0 \Leftrightarrow 1-\frac{r_n\alpha_n}{\alpha_0+\sum_{m\in\Omega_0}\alpha_m r_m}\geq 0$. Then, depending on the sign of $\beta_n$, the conclusion follows.



\end{proof}

Proposition~\ref{prop:biasUtility} means that the prosumers may have incentives to play strategically with the valuations of the bilateral trading prices with the root node since it influence the bias in their expected utility. More precisely, $\zeta_{n0}$ smaller than $\zeta_{0n}$ will lead to small bias values; whereas $\zeta_{n0}$ larger than $\zeta_{0n}$ will lead to large bias values. In order to minimize her bias, the prosumer would choose smaller valuation for the trade with the root node than the root node would choose for similar trade. 

\begin{proposition}\label{prop:UB}
There exists an upper-bound $\Phi_n$ such that for any $\bm{D}, \bm{G}, \bm{q} \in \text{SOL}^{\text{GNEP}}$, $\left|\mathbb{E}\Big[  F_n(\Pi_n)-\Pi_n  \Big]\right| \leq \Phi_n, \forall n \in \mathcal{N}$.
\end{proposition}

\begin{proof}

Taking the absolute value of the expectation of the difference between the estimated and the true utilities, we observe that 
$\left|\mathbb{E}\Big[F_n(\Pi_n)-\Pi_n \Big] \right| \leq \mathbb{E}\Big[ \left|F_n(\Pi_n)-\Pi_n\right| \Big] \leq \frac{1}{2}\left|\beta_n\right|\rho_n(\mathbf{r})^2$. 

Derivating $\rho_n(\mathbf{r})^2$ with respect to $r_n$, we obtain $\frac{\partial \rho_n(\mathbf{r})^2}{\partial r_n} = 2\frac{\rho_n(\mathbf{r})^2}{r_n}(1-\frac{r_n\alpha_n}{\alpha_0+\sum_{m\in\Omega_0}\alpha_m r_m}) \geq 0 $, which implies that $\rho_n(\mathbf{r})^2$ is increasing in $r_n$. Derivating $\rho_n(\mathbf{r})^2$ with respect to $r_m, m \neq n$, we obtain $\frac{\partial \rho_n(\mathbf{r})^2}{\partial r_m} = 2\rho_n(\mathbf{r})\Big[ -\frac{\rho_n(\mathbf{r})}{r_n}(\frac{1}{2\tilde{a}_m}+\frac{1}{a_m})\Big] \leq 0 $, which implies that $\rho_n(\mathbf{r})^2$ is decreasing in $r_m, m \neq n$. 

According to the parameterized variational inequality approach of Nabetani et al. \cite{nabetani2009}, the GNE set $\bm{D}, \bm{G}, \bm{q} \in \text{SOL}^{\text{GNEP}}$ can be described by making the valuation ratio $r_n$ span values in a certain interval, i.e., $\underline{r}_n\leq r_n \leq \overline{r}_n$, for any $n \in \Omega_0$.

Based on the variational analysis of $\rho_n(\mathbf{r})^2$, we conclude that $\Phi_n := \frac{1}{2}\left|\beta_n\right|\rho_n\Big(\overline{r}_n,(\underline{r}_{m})_{m\neq n}\Big)^2$.

\end{proof}

\subsection{Dealing with Privacy in the Three Nodes Network}

We still consider the three nodes example introduced in Subsection~\ref{ex:ex1}. The demand and the RES-based generation errors in the estimations are generated according to Gaussian density functions centered in $0$, with standard deviations:
\begin{equation*}
\bm{\sigma}^D=\left({\begin{array}{ccc} 0.8 & 0.2 & 0.2 \\
0.3 & 0.8 & 0.8 \\
0.8 & 0.1 & 0.3
\end{array}}\right) ,\ \ \bm{\sigma}^G=\left({\begin{array}{ccc} 0.0 & 0.2 & 0.5 \\
0.0 & 0.3 & 0.5 \\
0.0 & 0.8 & 0.10
\end{array}}\right), \text{ and }  \bm{Cov}=\left({\begin{array}{ccc} 0.4 & -0.2 & -0.3 \\
-0.8 & -1.0 & 0.5 \\
1.0 & 0.0 & 1.0
\end{array}}\right) \ . 
\end{equation*}
In Figure~\ref{fig:Bias}, we have represented the sum of the upper bounds on the estimated utility bias $\Phi_1+\Phi_2$ as a function of the prosumer utility parameters $\tilde{a}_1, \tilde{a}_2$. We have assumed that the maximum usage benefit is the same on nodes $1$ and $2$, leading to $\tilde{b}_1=\tilde{b}_2=60$. Based on the prosumer utility definition, the targeted demand is inversely proportional to the usage benefit parameter, indeed $D_n^*=\sqrt{\frac{\tilde{b}_n}{\tilde{a}_n}}$. We observe that depending on the utility parameters, the social welfare bias varies between $1.2 \%$ and $3.6 \%$. Furthermore, the bias is minimal when both prosumers have identical utility parameters (or equivalently, same target demands), and maximized when the utility parameters are asymmetric (one having a large target demand and the other a small one).

\begin{figure}[ht!]
\begin{center}
\includegraphics[scale=0.65]{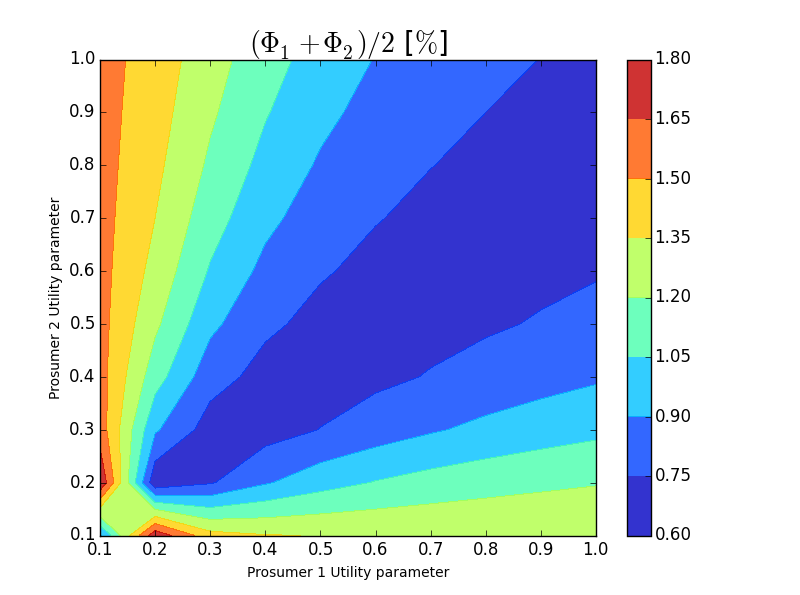}    
\caption{\small{Impact of the prosumer utility parameters $\tilde{a}_1, \tilde{a}_2$ on the social welfare bias.}}
\label{fig:Bias}
\end{center}
\end{figure}

\section{Conclusion}\label{sec:conclusion}

On a radial network where nodes are made of prosumers with differentiation price preferences, we formulate two market designs: \textbf{(i)} a centralized market design used as a benchmark, where a global market operator optimizes the flows (trades) and bilateral trading prices between the nodes to maximize the system overall social welfare; \textbf{(ii)} a fully distributed peer-to-peer market design where prosumers in local energy communities optimize selfishly the trades, demand, and flexibility activation in presence of private information. We characterize the solution of the peer-to-peer market as a Variational Equilibrium, without price arbitrage, and prove that the set of Variational Equilibria coincides with the set of social welfare optima solutions of market design \textbf{(i)}. 
In the presence of market imperfections, we propose a reformulation of the Generalized Nash Equilibrium Problem \textbf{(ii)} relying on the parameterized variational inequation approach of Nabetani et al. enabling the computation of the set of Generalized Nash Equilibria.
We also characterize formally the impact of preferences on the network line congestion and energy waste under both designs. The results are illustrated in two test cases (a three nodes network and the IEEE 14-bus network). We also provide a bound on the Price of Anarchy capturing the loss of efficiency caused by market imperfections in the three nodes example. We show that the impact of privacy on the social welfare bias is smaller than $3.6 \%$. 
Based on these performance analysis and numerical results, we conclude that peer-to-peer market design gives rise to similar performance than the classical centralized market design provided market imperfections (resulting from the lack of coordination, insufficient market liquidity, information asymmetry resulting from privacy) can be corrected, and constitutes a relevant evolution for power system operation as it promises more robustness and resilience. Indeed, as the information and decisions are not optimized by a  central single entity,  in case of failure or if one node is attacked, the power system can still rely on the other nodes. Besides, as all prosumers are  involved, they have the ability to adapt their actions to the state of grid. 

\newcommand{\url}[1]{\textcolor{blue}{#1}}

\end{document}